\documentclass[10pt]{amsart}
\usepackage{amsmath}
\usepackage{amssymb}
\usepackage{mathrsfs}
\usepackage{comment}
\usepackage{color}
\usepackage{todonotes}
\usepackage{cite}
\usepackage{hyperref}
\usepackage{mathabx}
\newtheorem{theorem}{Theorem}[section]
\newtheorem{obs}[theorem]{Observation}
\theoremstyle{definition}
\newtheorem{definition}[theorem]{Definition}
\newtheorem{corollary}[theorem]{Corollary}
\newtheorem{lemma}[theorem]{Lemma}
\newtheorem{proposition}[theorem]{Proposition}
\newtheorem{claim}[theorem]{Claim}
\theoremstyle{remark}
\newtheorem{remark}[theorem]{Remark}
\usepackage[all]{xy}
\usepackage{epic}
\newcommand{\sqdiamond}{\xymatrix{*[F]{\diamondsuit}}}
\newcommand{\sq}{\sqdiamond\!\!\hphantom{b}}
\renewcommand\square{\Box}

\newcommand\Zscr{\mathscr{Z}}
\newcommand{\Th}{{}^{\textrm{th}}}

\newcommand{\concat}{\vphantom{B}^\smallfrown}
\newcommand{\cat}{\concat}

\newcommand\Pbb{\mathbb{P}}
\newcommand\Rbb{\mathbb{R}}
\newcommand\Sbb{\mathbb{S}}
\newcommand\Zbb{\mathbb{Z}}
\newcommand\Qbb{\mathbb{Q}}

\newcommand\Lbf{\mathbf{L}}
\newcommand\Vbf{\mathbf{V}}
\newcommand\PFA{\mathrm{PFA}}
\newcommand{\restr}{\upharpoonright}
\newcommand{\forces}{\Vdash}
\newcommand\lex{\mathrm{lex}}
\newcommand\Cfrak{\mathfrak{C}}
\newcommand\Mfrak{\mathfrak{M}}

\newcommand\seq[1]{\langle #1  \rangle}
\newcommand\functions[2]{\!\vphantom{B}^{{#1}}#2}
\DeclareMathOperator{\cf}{cf}
\DeclareMathOperator{\h}{ht}
\DeclareMathOperator{\otp}{otp}
\DeclareMathOperator{\last}{last}
\DeclareMathOperator{\Top}{top}

\DeclareMathOperator{\r2}{\varrho_2}
\DeclareMathOperator{\nacc}{nacc}
\DeclareMathOperator{\acc}{acc}
\numberwithin{equation}{section}

\author[Cummings]{James Cummings}
\author[Eisworth]{Todd Eisworth}
\author[Moore]{Justin Tatch Moore}
\title{On minimal non-$\sigma$-scattered linear orders}



\thanks{The authors would like to thank Assaf Rinot for reading an earlier version of the paper
and pointing out that $\sq_\kappa^{+\epsilon}$ is equivalent to $\sq_\kappa$.
A proof of this fact has been included with his permission.
They would also like to thank the anonymous referee for
their careful reading and useful comments.
The first author's research on this project is supported in part by NSF grant DMS--2054532.
The third author's research on this project is supported in part by NSF grants DMS--1854367 and DMS--2153975.}

\address{Todd Eisworth \\ Department of Mathematics \\ Ohio University \\ Athens, OH 45701--2979}
\address{James Cummings \\ Mathematical Sciences Department \\ Carnegie Mellon University \\ Pittsburgh, PA 15213-3890}
\address{Justin Tatch Moore \\ Department of Mathematics \\ 310 Malott Hall \\ Cornell University \\ Ithaca, NY 14853--4201}
\begin{document}

\begin{abstract}
The purpose of this article is to give new constructions of linear orders which are \emph{minimal} with respect to being \emph{non-$\sigma$-scattered}.
Specifically, we will show that Jensen's principle $\diamondsuit$ implies that there is a minimal Countryman line,
answering a question of Baumgartner \cite{deepcut}.  We also produce the first consistent examples of minimal non-$\sigma$-scattered linear orders of cardinality greater than $\aleph_1$, as given a successor cardinal $\kappa^+$, we obtain such linear orderings of cardinality $\kappa^+$ with the additional property that their square is the union of $\kappa$-many chains.  We give two constructions:  directly building such examples using forcing, and also deriving their existence from combinatorial principles.  The latter approach shows that such minimal non-$\sigma$-scattered linear orders of cardinality $\kappa^+$ exist for every cardinal $\kappa$ in G\"odel's constructible universe, and also (using work of Rinot~\cite{rinot}) that examples must exist at successors of singular strong limit cardinals in the absence of inner models satisfying the existence of a
measurable cardinal $\mu$ of Mitchell order $\mu^{++}$. 
\end{abstract}

\maketitle
\section{Introduction}
The class $\Mfrak$ of \emph{$\sigma$-scattered linear orders} was considered by Galvin as a natural generalization of the classes of countable linear orders and well orders.
On the one hand $\Mfrak$ is quite rich, and on the other it is amenable to refined structural analysis.
Recall that a linear order is \emph{scattered} if it does not contain a copy of the rational line $(\Qbb,\leq)$ and
is \emph{$\sigma$-scattered} if  it is a union of countably many scattered suborders.
Both of these classes include the well orders and are closed under lexicographic sums $\sum_{i \in K} L_i$ and the converse operation $L \mapsto L^*$ which reverses the order on $L$; in fact Hausdorff \cite{hausdorff_scattered} showed that the scattered orders form the least class with these closure  
properties.

The $\sigma$-scattered orders form the least class with these closure properties and the additional property of
closure under countable unions.
In \cite{FraisseConj}, Laver proved Fra{\"i}ss\'e's conjecture that the countable linear orders are
\emph{well quasi-ordered}: whenever
$L_i$ $(i < \infty)$ is a sequence of countable linear orders, there is an $i < j$ such that $L_i$ embeds into $L_j$.
In fact, his proof established the following celebrated result.
\begin{theorem} (Laver \cite{FraisseConj})
The class $\Mfrak$ is
well quasi-ordered by the embeddability relation.
\end{theorem}

Empirically, $\Mfrak$ is the largest class of linear orders which is immune to set-theoretic independence phenomena.
It is therefore natural to study those linear orders which lie just barely outside of $\Mfrak$.
In general, given a class $\Cfrak$ of linear orders, we will say that a linear order $L$ is a \emph{minimal} element of $\Cfrak$ if $L$ is in $\Cfrak$ and embeds into all of its suborders which are in $\Cfrak$. 
In this paper we will investigate those linear orders $L$ which are minimal with respect to not being in $\Mfrak$.
More precisely, we will prove that it is consistent that for each infinite cardinal $\kappa$,
there is a linear order of cardinality $\kappa^+$ which is minimal with respect to being non-$\sigma$-scattered.
Previously it was not known if it was consistent to have a minimal non-$\sigma$-scattered order of cardinality
greater than $\aleph_1$.
Moreover, even our construction of a minimal non-$\sigma$-scattered order of cardinality $\aleph_1$ is novel and 
answers a question of Baumgartner \cite[p. 275]{deepcut}.

\subsection*{Mathematical and historical background}
One of the first results on scattered linear orders is the following result of Hausdorff.

\begin{theorem} \label{hausdorff_theorem} (Hausdorff \cite{hausdorff_scattered}, see also \cite{rosenstein})
If $\kappa$ is a regular cardinal and $L$ is a scattered linear order of cardinality $\kappa$,
then either $\kappa$ or $\kappa^*$ embeds into $L$.
\end{theorem}

While $\sigma$-scattered linear orders were not considered until \cite{FraisseConj}, Theorem \ref{hausdorff_theorem}
immediately generalizes to the class of $\sigma$-scattered linear orders.
Since neither $\omega_1$ nor $\omega_1^*$ embed into $\Rbb$, it follows that no uncountable set
of reals is $\sigma$-scattered.
For brevity, we will say that a linear order is a \emph{real type} if it is isomorphic to an uncountable
suborder of the real line.

The properties of real types are already sensitive to set theory.
On one hand, a classical diagonalization argument yields the following result of Dushnik and Miller.

\begin{theorem} \label{dushnik-miller} (Dushnik and Miller \cite{dushnik-miller})
Assume CH.
For any uncountable $X \subseteq \Rbb$ there is an uncountable $Y \subseteq X$ such that $Y^2$ does not contain
the graph of any uncountable strictly monotone function other than the identity
(and hence does not embed into any proper suborder).
\end{theorem}

On the other hand, Baumgartner demonstrated
that if $X, Y \subseteq \Rbb$ are $\aleph_1$-dense\footnote{
A linear order is \emph{$\kappa$-dense} if it has no first or last elements and each interval has cardinality $\kappa$.}
and CH holds, then there is a c.c.c. forcing which makes $X$ and $Y$ order isomorphic \cite{reals_iso}.
In particular, he showed that there is always a forcing extension in which every two $\aleph_1$-dense sets of reals
are isomorphic.
This result is now often phrased axiomatically as follows.

\begin{theorem} (Baumgartner \cite{reals_iso}) \label{reals_iso}
Assume PFA. 
Any two $\aleph_1$-dense subsets of $\Rbb$ are isomorphic.
In particular, any real type of cardinality $\aleph_1$ is minimal.
\end{theorem}

Here the Proper Forcing Axiom (PFA) is a powerful generalization of the Baire Category
Theorem.
It plays an important role in the broader analysis of non-$\sigma$-scattered linear orders as we will see momentarily.
More information on PFA in the context of linear orders can be found in \cite{walksbook}; see e.g. \cite{PFA:abraham}, \cite{PFA:baumgartner}, \cite{PFA:moore}, \cite{FA:todorcevic} for an introduction to PFA and its consequences.

Another class of non-$\sigma$-scattered linear orders is provided by the \emph{Aronszajn lines}:\footnote{Aronszajn
lines are also known as \emph{Specker types}.} uncountable
linear orders with the property that they do not contain uncountable suborders which are either separable or
scattered.
Aronszajn lines were first constructed by Aronszajn and Kurepa (see \cite{trees:Kurepa} \cite{trees:Todorcevic})
in the course of analyzing Souslin's Problem \cite{souslin}, which asks if $\Rbb$ is the only complete dense linear
order in which every family of pairwise disjoint intervals is countable.
By Theorem \ref{hausdorff_theorem}, Aronszajn lines are necessarily non-$\sigma$-scattered.

In the 1970s, R.~Countryman introduced a class of linear orders now known as \emph{Countryman lines}.
These are the uncountable linear orders $C$ such that $C \times C$ is a union of countably many chains.
Such orders are necessarily Aronszajn and have the property that no uncountable linear order can embed
into both $C$ and $C^*$.
They were first constructed by Shelah \cite{countryman_shelah}, with a simplified construction later being given by Todorcevic \cite{acta}.
Notice that being Countryman is clearly inherited by uncountable suborders.

Abraham and Shelah proved the analog of Theorem \ref{reals_iso} for Countryman lines.

\begin{theorem} (Abraham and Shelah \cite{club_iso})
Assume PFA.
Any Countryman line embeds into all of its uncountable suborders.
Moreover, any two regular\footnote{
An Aronszajn line $L$ is \emph{regular} if $L$ is $\aleph_1$-dense and the collection of all countable subsets of $L$
which are closed in the order topology contains a closed and cofinal set in $([L]^\omega,\subset)$.
}
Countryman lines are either isomorphic or reverse isomorphic.
\end{theorem}

The next results give a complete classification of the Aronszajn lines under PFA.

\begin{theorem} (Moore \cite{linear_basis})
Assume PFA.
Every Aronszajn line has a Countryman suborder.
\end{theorem}

\begin{theorem} (Martinez-Ranero \cite{A-line_WQO})
Assume PFA.
The Aronszajn lines are well quasi-ordered by embeddability.
\end{theorem}

The next theorem gives a complete characterization of the minimal non-$\sigma$-scattered linear orders
under $\PFA^+$, a strengthening of PFA.

\begin{theorem} (Ishiu and Moore \cite{ishiu-moore})
Assume $\PFA^+$.
Every minimal non-$\sigma$-scattered linear order is isomorphic to either a set of reals of cardinality $\aleph_1$
or a Countryman line.
Furthermore, any non-$\sigma$-scattered linear order contains a non-$\sigma$-scattered suborder of cardinality
$\aleph_1$.
\end{theorem}

Since PFA and $\PFA^+$ are rather strong assumptions,
it is natural to ask what is possible in other models of set theory.
While it is reasonable to think that some enumeration principle such as CH or $\diamondsuit$ might allow
one to prove an analog of Theorem \ref{dushnik-miller} for Aronszajn lines, Baumgartner showed that this is not the case
(Baumgartner's construction contained an error which was later corrected by D.~Soukup).

\begin{theorem} (Baumgartner \cite{deepcut}, D.~Soukup \cite{soukup})
Assume $\diamondsuit^+$.
There is a Souslin line which embeds into all of its uncountable suborders.
\end{theorem}

Here a \emph{Souslin line} is a nonseparable linear order in which every family of pairwise disjoint
intervals is countable.
Any Souslin line can be embedded in a Souslin line which is moreover dense and complete as a linear order---
hence the existence of a Souslin line is equivalent to the existence of a counterexample to Souslin's Problem.
On the other hand, any Souslin line $L$ contains a suborder which is Aronszajn---simply pick a sequence of 
points $\{x_\alpha : \alpha < \omega_1\}$ from $L$ such that for all $\beta < \omega_1$, $x_\beta$
is not in the closure of $\{x_\alpha : \alpha < \beta\}$.
Furthermore it is easily checked that
Aronszajn suborders of Souslin lines are themselves Souslin.

While Baumgartner's construction produces a minimal Aronszajn line, it should be noted that
Souslin lines are necessarily not Countryman.
In \cite{deepcut}, Baumgartner asked if $\diamondsuit^+$ could be weakened to $\diamondsuit$ in his construction and if his argument could be adapted to construct a minimal Aronszajn line which was not Souslin.

In \cite{w1_w1*_min}, the third author proved that it is consistent that there are no minimal Aronszajn lines.
This was achieved by obtaining a model of CH which also satisfied a certain combinatorial consequence of PFA.
That CH held in this model also yielded the following stronger result.
\begin{theorem} (Moore \cite{w1_w1*_min})
It is consistent (with CH) that $\omega_1$ and $\omega_1^*$ are the only minimal uncountable linear orders.
\end{theorem}
In \cite{soukup2}, D.~Soukup adapts this argument to
show that the existence of a Souslin line does not imply the existence of a minimal Aronszajn line.

The strategy in \cite{w1_w1*_min}
was combined with the analysis of \cite{ishiu-moore} to yield the following result.

\begin{theorem} (Lamei Ramandi and Moore \cite{no_min_non-sigma-scat})
If there is a supercompact cardinal,
there is a forcing extension in which CH holds and
there are no minimal non-$\sigma$-scattered linear orders.
\end{theorem}

On the other hand, Lamei Ramandi has shown that $\diamondsuit$ is consistent with the existence of
a minimal non-$\sigma$-scattered linear order
which is neither a real nor Aronszajn type.
In fact he has produced two qualitatively different constructions.

\begin{theorem} (Lamei Ramandi \cite{lamei-ramandi})
It is consistent with $\diamondsuit$
that there is a minimal non-$\sigma$-scattered linear order $L$ with cardinality $\aleph_1$
which is a dense suborder of a Kurepa line\footnote{A \emph{Kurepa line} is a linear order of density $\aleph_1$ which has cardinality greater
than $\aleph_1$ and does not contain a real type.
See \cite{trees:Todorcevic} for more information.}.
\end{theorem}

\begin{theorem} (Lamei Ramandi \cite{lamei-ramandi2})
It is consistent with $\diamondsuit$ that there is a minimal non-$\sigma$-scattered order with the property that every
uncountable suborder contains a copy of $\omega_1$.
\end{theorem}

\subsection*{Main results}
Up to this point though, all consistent examples of minimal non-$\sigma$-scattered linear orders
are of cardinality $\aleph_1$.
In order to state our main result, we need to introduce another definition.
A linear order $L$ is \emph{$\kappa^+$-Countryman} if  $L$ has cardinality $\kappa^+$ and $L^2$ is a union of $\kappa$ chains.

\begin{theorem} \label{main_thm}
Assume $\Vbf = \Lbf$.
For each infinite cardinal $\kappa$, there is a $\kappa^+$-Countryman line which is minimal with respect to being non-$\sigma$-scattered.
\end{theorem}
\noindent
In fact, the construction in Theorem \ref{main_thm} factors
through the combinatorial principle $\sqdiamond_\kappa$ considered in \cite{ASS}, \cite{rinot}, \cite{rinotschindler}.
This has added interest, because the main result of Rinot's \cite{rinot} shows that if $\mu$ is a singular cardinal, the principle $\sq_\mu$ is equivalent to the conjunction of $\square_\mu$ and $2^{\mu}=\mu^+$.  As a corollary, it follows that the failure of $\sq_\mu$ at a singular strong limit cardinal has large cardinal strength: in such a situation, either there is a violation of the singular cardinal hypothesis or $\square_\mu$ fails.
Either of these possibilities carries large cardinal strength \cite{CONnegSCH1}, \cite{CONnegSCH2}, \cite{K_wo_measurable},
\cite{square_core}
and so we obtain, for example, the following striking corollary.
  
\begin{corollary}
Assume that there is no inner model which safisfies there a measurable cardinal $\mu$ of Mitchell order $\mu^{++}$.
If $\kappa$ is the successor of a singular strong limit cardinal, then there is a minimal non-$\sigma$-scattered
linear order of cardinality $\kappa$.
\end{corollary}

The  $\kappa^+$-Countryman lines are interesting in their own right.
Although they have almost exclusively been studied when $\kappa = \aleph_0$ (in which case they are known simply as \emph{Countryman lines}),
their remarkable properties readily generalize to the higher cardinal case:
\begin{itemize}

\item $\kappa^+$-Countryman lines do not contain a copy of $\kappa^+$ or its converse.

\item If $L$ is $\kappa^+$-Countryman and $X \subseteq L$ has cardinality $\kappa^+$, then
there is a family of pairwise disjoint intervals of $X$ of cardinality $\kappa^+$.
In particular $X$ has density $\kappa^+$.

\item If $L$ is $\kappa^+$-Countryman, then no linear order of cardinality $\kappa^+$ embeds into
both $L$ and $L^*$.

\end{itemize}

Our argument for $\kappa = \aleph_0$ is somewhat simpler and of independent interest
as it answers Baumgartner's question mentioned above.

\begin{theorem} \label{baum_sol}
Assume $\diamondsuit$.
There is a Countryman line which embeds into all of its uncountable suborders.
\end{theorem}

\subsection*{Organization}
Section \ref{prelim:sec} will contain a review of the basic analysis of trees and linear orders which we will need.
In Section \ref{baumgartner:sec}, we will show that $\diamondsuit$ is sufficient to construct a minimal Countryman line.
Section \ref{higher_coherence} contains the basic analysis of $\kappa^+$-Countryman lines for arbitrary infinite cardinals $\kappa$.
A framework for constructing $\kappa^+$-Countryman lines which are minimal with respect to being non-$\sigma$-scattered
is introduced in Section \ref{S_kappa:sec}.
This framework is then put to use in Sections  \ref{forcing:sec} and \ref{axiomaic:sec} where we present forcing
and axiomatic constructions of such linear orders.
Finally, Section \ref{limitations:sec} contains some concluding remarks.

\section{Preliminaries}

\label{prelim:sec}

We will begin with a brief review of some notation, terminology, and concepts from set theory which we will need.
None of the material in this section is new or due to the authors.
Further information on trees and linear orders can be found in \cite{rosenstein} and \cite{trees:Todorcevic}.
Both \cite{set_theory:Jech} and \cite{set_theory:Kunen} are standard references for set theory
(\cite{set_theory:Jech} is encyclopaedic whereas \cite{set_theory:Kunen} is more detail oriented).

All counting starts at $0$.
As is standard, we will use $\omega$ to denote the set of finite ordinals, which we take to coincide with the nonnegative integers.
A {\em sequence} is a function whose domain is an ordinal.
The domain of a sequence $s$ is typically referred to as its {\em length} and denoted $|s|$.
If $s$ and $t$ are sequences of ordinals, we define $s \leq_{\lex} t$ if either $s$ is an initial part of $t$ or else there
is a $\xi < \min(|s|,|t|)$ with $s(\xi) \ne t(\xi)$
and, for the least such $\xi$, $s(\xi) < t(\xi)$.
We will generally identify a function with its graph.
In particular if $f$ and $g$ are functions, $f \subseteq g$ exactly when $f$ is a restriction of $g$ (including the possibility $f=g$).

If $s$ and $t$ are two sequences taking values in $\Zbb$, we define $s+t$ to be the sequence of
length $\min(|s|,|t|)$ obtained by adding $s$ and $t$ coordinatewise on the restricted domain.
If $t$ is a sequence taking values in $\Zbb$, then $-t$ is the sequence of length $|t|$ obtained by multiplying $t$
coordinatewise by $-1$.    
As is standard, $s-t$ abbreviates $s + (-t)$.

We note that any linear ordering is isomorphic to a set of sequences of ordinals ordered by $\leq_\lex$.
If one closes this set of sequences under initial segments, the structure of this set equipped with the extension partial order captures
important aspects of the linear order.
For this reason, it is fruitful to abstract this concept.
A {\em tree} is a partially ordered set $(T, \leq_T)$ in which the set $\{s\in T: s<_T t\}$ of predecessors of $t$ is well-ordered by $<_T$
for any $t\in T$. 
The order-type of this set is called the {\em height of $t$}.  The collection of all elements of $T$ of a given height $\delta$ will be denoted by $T_\delta$, referred to as the {\em $\delta\Th$ level of $T$}.  The {\em height} of the tree $T$ is the least $\delta$ such that $T$ contains no elements of height $\delta$. 
Notation such as $T_{\leq\delta}$ should be given the obvious interpretation.
If $\kappa$ is an infinite cardinal, a tree is a {\em $\kappa$-tree} if the height of $T$ is $\kappa$ and all levels of $T$ have cardinality less than $\kappa$.

We say that $T$ is {\em Hausdorff} if whenever $s,t \in T$ have limit height and are distinct,
they have distinct sets of predecessors.
If $T$ is a set of sequences which is downwards closed with respect to $\leq$, then $(T,\leq)$ is a Hausdorff tree
and moreover $T_\alpha$ consists of the sequences in $T$ of length $\alpha$;
we say that $T$ is a {\em tree of sequences}.
Conversely, any Hausdorff tree is isomorphic to a tree of sequences.

In this paper we will work with trees of sequences which moreover have the property that sequences of limit length
$\delta$ are extended by a unique element of the tree of length $\delta+1$.
For this reason, we will typically work with trees consisting of sequences of successor length and which
are closed under taking initial segments of successor length.

An {\em antichain} in a tree $T$ is a collection of pairwise incomparable elements.
It is worth noting that in a tree if $s$ and $t$ are incomparable, they have no common upper bound
(i.e. they are \emph{incompatible}).
If $T$ is a tree, $S$ is a {\em subtree}\footnote{This meaning of ``subtree'' and ``branch'' are not completely standard.}
of $T$ if $S \subseteq T$, $S$ is downward closed in $T$, and $S$ has the same height as $T$.
A subtree of $T$ which is a chain is a {\em branch} of $T$.

A {\em $\kappa$-Aronszajn tree} is a $\kappa$-tree with no branches.
A linear ordering $L$ is a {\em $\kappa$-Aronszajn line} if it does not contain a copy of $\kappa$ or $\kappa^*$
and whenever $X \subseteq L$
has cardinality $\kappa$, its density is $\kappa$.
It is a standard fact that the lexicographic ordering on a $\kappa$-Aronszajn tree is a $\kappa$-Aronszajn line and any $\kappa$-Aronszajn
line is isomorphic to the lexicographic ordering of some subset of a $\kappa$-Aronszajn tree of sequences.
If $\kappa = \aleph_1$, then we just write ``Aronszajn'' instead of ``$\kappa$-Aronszajn.''

(Note that by Hausdorff's theorem, any linear order which does not contain $\kappa$ or $\kappa^*$ also does not
contain any scattered suborders of cardinality $\kappa$.
Hence this definition of \emph{Aronszajn line} is equivalent to the one given in the introduction.)
A linear order $C$ is $\kappa$-Countryman if $C$ has cardinality $\kappa$ and $C^2$ is a union of fewer than $\kappa$ chains (in the coordinatewise order);
we will write ``Countryman'' to mean ``$\aleph_1$-Countryman.''
The basic analysis of $\kappa$-Countryman lines can be found in Section \ref{higher_coherence}.

Finally, we recall a useful characterization of $\sigma$-scattered linear orders which follows easily from Galvin's analysis
of $\Mfrak$ (see \cite{FraisseConj}).
If $\gamma$ is an infinite ordinal, consider the collection $\Qbb_\gamma \subseteq \Qbb^\gamma$ consisting of all $x$ which change their values
finitely often:
there exist $0 = \xi_0 < \ldots < \xi_n = \gamma$
such that if $i < n$, $x$ is constantly $q_i$ on $[\xi_i,\xi_{i+1})$.
We equip $\Qbb_\gamma$ with the lexicographic order.
Since $|\Qbb_\gamma| = |\gamma|$, neither $\gamma^+$ nor its converse embed into $\Qbb_\gamma$. 
It is also easily checked by induction that any interval in $\Qbb_\gamma$ contains copies of $\delta$ and $\delta^*$
for any ordinal $\delta < \gamma^+$.
Thus by Theorem 3.3 of \cite{FraisseConj}, $\Qbb_\gamma$ is $\sigma$-scattered and
any $\sigma$-scattered linear order of cardinality $|\gamma|$ embeds into $\Qbb_{\gamma}$
(if $L$ is $\sigma$-scattered and has cardinality at most $|\gamma|$, then $L \times \Qbb_\gamma$ and $\Qbb_\gamma$
satisfy (i)--(iii) of \cite[3.3]{FraisseConj} for $\alpha = \beta = \gamma^+$ and hence are biembeddable).
Rephrasing this, we have the following.

\begin{proposition} \label{sigma_scat_char}
If $\gamma$ is an infinite ordinal, then a linear order of cardinality at most $|\gamma|$ is $\sigma$-scattered if and only
if it embeds into $\Qbb_\gamma$.
In particular $\Qbb_\gamma$ is biembeddable with $\Qbb_{|\gamma|}$.
\end{proposition}

\section{Baumgartner's Question}

 \label{baumgartner:sec}

In this section, we prove Theorem \ref{baum_sol},
thus answering Baumgartner's questions by showing that from $\diamondsuit$, one may construct a minimal 
Countryman line.
As already noted, such a linear order is Aronszajn but not Souslin. 

It will be useful to define some notation and terminology before proceeding.
\begin{definition} \label{Sbb:def}
$\Sbb$ is the set of all $s \in \functions{<\omega_1}{\omega}$ of successor length
which are finite-to-one.
\end{definition}
We will view $\Sbb$ as being equipped with the order of extension, making it a tree.
Define $f:\Sbb \to \omega \times \omega$ by 
\[
f(s) := (s(\xi) , |\{\eta < \xi : s(\eta) = s(\xi)\}|)
\]
where $|s| = \xi + 1$.
Observe that if $f(s) = f(s')$, then $s$ and $s'$ are incomparable.
Thus $\Sbb$ is special.

The next definitions abstract the aspects of the tree of sequences associated to
the function $\varrho_2$ of \cite{acta}.
\begin{definition}
A \emph{$\varrho_2$-modifier} is a continuous integer-valued sequence of successor length. 
If $s,t \in \Sbb$, we say that $s$ is a {\em $\varrho_2$-modification} of $t$ if $|s| = |t|$ and $s-t$ is a $\varrho_2$-modifier.
If $X \subseteq \Sbb$, we will say that $X$ is {\em closed under $\varrho_2$-modifications} if whenever
$s \in X$ and $t$ is a $\varrho_2$-modification of $s$, $t \in X$.
We say $X$ is {\em $\varrho_2$-full} if it is uncountable and closed under initial segments
of successor length and $\varrho_2$-modifications.
\end{definition}

We will sometimes drop the prefix ``$\varrho_2$-'' from $\varrho_2$-modifier for brevity.
Notice that if $m$ is a $\varrho_2$-modifier of length $\alpha+1$ for $\alpha$ limit, then $m$ is uniquely determined by
its restriction to $\alpha$.
In fact any continuous sequence $s$ taking values in $\Zbb$ that is eventually constant and having limit length
can be uniquely extended to a $\r2$-modifier of length $|s| + 1$.
It will sometimes be useful to regard such sequences $s$ as $\varrho_2$-modifiers by identifying them with the minimum
modifier which extends them.

Proposition \ref{sigma_scat_char} yields the following proposition.

\begin{proposition} \label{mod_sigma-scat}
For any successor ordinal $\alpha$, the set of modifiers $s$ of length $\alpha$ is $\sigma$-scattered when ordered by $\leq_\lex$.
\end{proposition}

\begin{definition} \label{rho2_coherent:def}
A subset $X \subseteq \Sbb$ is {\em $\varrho_2$-coherent} if whenever $s,t \in X$ with $|s| \leq |t|$,
$t \restriction |s|$ is a $\varrho_2$-modification of $s$. 
For ease of reading we will write ``full'' instead of ``$\varrho_2$-full'' in the context of ``$\varrho_2$-coherent.''
\end{definition}

Notice that there are only countably many $\varrho_2$-modifications of an element of $\Sbb$ and therefore
any $\varrho_2$-coherent full subset of $\Sbb$ is a subtree which has countable levels and hence is an Aronszajn tree.
The following theorem is essentially due to Todorcevic (see \cite[3.4]{acta}); 
see the proof of the more general Proposition \ref{kappa-Countryman} below.

\begin{theorem} \label{countryman_trans}
If $C \subseteq \Sbb$ is uncountable and $\varrho_2$-coherent, then $(C,\leq_\lex)$ is Countryman.    
\end{theorem}

We will prove Theorem \ref{baum_sol} by showing that
$\diamondsuit$ implies the existence of a full $\varrho_2$-coherent tree $T \subseteq \Sbb$
with the property that for any uncountable antichain $X$ of $T$ and any uncountable subset $Y$ of $T$, there is an embedding of $(X, \leq_\lex)$ into $(Y, \leq_\lex)$.

\begin{lemma}
\label{embedsubtrees}
Suppose $T$ is a $\varrho_2$-coherent subtree of $\mathbb{S}$ such that for any subtree $S$ of $T$ there is an embedding $\phi:T\rightarrow S$ that preserves both the lexicographic order $\leq_\lex$ and incompatibility with respect to $T$'s tree order.
Then given any uncountable antichain $X$ of $T$ and uncountable subset $Y$ of $T$, there is an embedding of $(X,\leq_\lex)$ into $(Y,\leq_\lex)$.
\end{lemma}
\begin{proof}
Observe that by replacing $Y$ with an uncountable subset if necessary, 
we may assume $Y$ is an antichain in $T$.
Let $S$ be the downward closure of $Y$ in $T$ and let $\phi:T\rightarrow S$ be the hypothesized embedding.
We define a function $f:X\rightarrow Y$ by letting $f(x)$ be some element of $Y$ that extends $\phi(x)$;
this is possible by our choice of $S$.
Given $x<_\lex y$ in $X$, we know that $\phi(x)$ and $\phi(y)$ must be incompatible in $S$, and 
$\phi(x)<_\lex \phi(y)$.
But this implies 
$f(x)<_\lex f(y)$ 
as well, and we are done.
\end{proof}

How does this previous lemma help our project?  It tells us that it will be sufficient to build a $\varrho_2$-coherent
$T$ that admits suitable embeddings into any of its subtrees. 
Our strategy is to build $T$ so that every subtree will contain a tree of a canonical form that
will render the existence of the required $\phi$ obvious.
This provides the motivation for the next set of definitions, which capture a crucial ingredient in our proof.  

\begin{definition}
Suppose $n<\omega$ and $s$ and $t$ are in $\mathbb{S}$.
We say that {\em $t$ is an $n$-extension of $s$}, written $s\subseteq_n t$,  if  $s \subseteq t$ and whenever
$|s| \leq \xi < |t|$, $t(\xi) \geq n$.
\end{definition}

\begin{definition}
Suppose that $T \subseteq \Sbb$ is $\varrho_2$-coherent and full.
\begin{enumerate}
\item 
The {\em cone of $T$ determined by $s$}, denoted $T[s]$, is defined as usual by
\[
    T[s]:=\{t\in T: t\subseteq s\text{ or }s\subseteq t\}.
\]
\item 
The  {\em frozen cone of $T$ determined by $s$ and $n$}, denoted $T[s, n]$, is defined by
\[
    T[s, n]:=\{t\in T: t\subseteq s\text{ or }s \subseteq_n t\}.
\]
\item 
Given an ordinal $\delta$, we let $T_\delta[s]$ denote the elements of $T[s]$ of height $\delta$, and similarly for $T_\delta[s, n]$.
\end{enumerate}
\end{definition}

It is clear that any cone of $T$ is also a frozen cone, as $T[s]$ is just $T[s, 0]$.
Since $T$ is $\varrho_2$-full,
frozen cones of $T$ are also subtrees of $T$.

\begin{lemma}
\label{operations}
If $T \subseteq \Sbb$ is $\varrho_2$-coherent and full, 
then for any $s\in T$ and $n<\omega$ there is an embedding $\phi$ of $T$ into $T[s, n]$ that preserves the lexicographic order $\leq_\lex$ and incompatibility with respect to the tree ordering.
\end{lemma}

Before we begin the proof of the lemma, it will be useful to introduce two operations on $\Sbb$. 
\begin{definition}
If $s,t, u \in \Sbb$, $|s|< |t| = |u|$, and   
\[
u(\xi):=
\begin{cases}
s(\xi) &\text{if $\xi < |s|$, and}\\
t(\xi) &\text{if $|s| \leq \xi < |t|$}
\end{cases}
\]
then we say that $u$ is obtained by {\em writing $s$ over $t$}.
\end{definition}

\begin{definition}
If $t\in T$, $\beta< |t|$, and $n<\omega$, then the sequence $v$ defined by
\[
v(\xi):=
\begin{cases}
t(\xi)  &\text{if $\xi<\beta$, and}\\
t(\xi)+n  &\text{if $\beta \leq \xi < |t|$}
\end{cases}
\]
is {\em the result of translating $t$ by $n$ beyond $\beta$}.
\end{definition}

\begin{proof}[Proof of Lemma \ref{operations}.]
Observe that if $T \subseteq \Sbb$ is $\varrho_2$-coherent and full, then it is closed under these two 
operations.
We prove the lemma in two stages.
First, we prove that for any $s\in T$ and $n<\omega$ there is such an embedding from the (ordinary) cone $T[s]$ into the frozen cone $T[s,n]$.
Doing this is straightforward: given $t\in T[s]$ extending $s$, we translate $t$ by $n$ beyond $|s|$.
This function has the required properties, and by our assumption on $T$ the range is contained in $T[s, n]$.

Next, we show for any $s\in T$ that $T$ can be embedded into the (ordinary) cone $T[s]$ preserving the lexicographic order and incompatibility. 
To do this, define
$\delta := |s|+\omega+1$. 
Observe that $(T_\delta[s],<_\lex)$ is a dense linear order.
Since $T_{\leq \delta}$ is countable, there is a $\leq_\lex$-preserving embedding
\[
\phi_0:T_{\leq \delta}\rightarrow T_\delta[s].
\]
\noindent
Notice that $\phi_0$ trivially preserves incompatibility since $T_\delta[s]$ is an antichain.  We extend $\phi_0$ to a function $\phi:T\rightarrow T[s]$ by letting $\phi(t)$ be the result of writing $\phi_0(t\restr\delta)$ over $t$ for $t$ of height greater than $\delta$.   Again, our assumptions imply that the range of $\phi$ is contained in $T[s]$, and the function preserves both $\leq_\lex$ and incompatiblity.
\end{proof}

Now we come to the point: if we can build a tree as in Lemma~\ref{operations} with the property that any subtree contains a frozen cone, then we will have what we need to establish Theorem~\ref{baum_sol}.

\begin{proposition}
    Suppose that $T \subseteq \Sbb$ is $\varrho_2$-coherent, full and has the property that every
    subtree of $T$ contains a frozen cone.
    If $C \subseteq T$ is any uncountable antichain, then $(C,\leq_\lex)$ is a minimal Countryman line.
\end{proposition}

We have therefore reduced our task to establishing the following proposition.
\begin{proposition}
    Assume $\diamondsuit$.
    There is a $T \subseteq \Sbb$ which is $\varrho_2$-coherent, full, and has the property that every 
    subtree of $T$ contains a frozen cone.
\end{proposition}

We will pause to introduce some notation and terminology which, while a little gratuitous now,
anticipates the greater complexities of the higher cardinal constructions in later sections.
Let $\equiv$ denote the equivalence relation on $\Sbb$ defined by $s \equiv t$ if $t$ is a $\varrho_2$-modification of $s$.
Define $\mathbb{P}:=\{[s]:s\in\mathbb{S}\}$
to be the collection of all $\equiv$-equivalence classes of functions in $\mathbb{S}$, and order $\mathbb{P}$ in the natural way: 
given $q$ and $p$ in $\mathbb{P}$, we define $q\leq_{\mathbb{P} }p$ to mean that some element of $q$ extends some element of $p$ in $\mathbb{S}$. 
We extend the notion of ``height'' to elements of $\mathbb{P}$ in the obvious way:
the height of $p$ is the height of any of its elements.

Our construction will depend on the interplay between the partially ordered sets
$(\mathbb{S}, \supseteq)$ and $(\mathbb{P}, \leq_{\mathbb{P}})$,
and we explore that relation a little with the following observations.
We start by recording some easy facts  
about the interaction between the equivalence relation $\equiv$ and the operations on sequences
from Lemma \ref{operations}. 

\begin{lemma} \label{operations_and_ER} 
The following are true:
\begin{enumerate}
\item Let $s_0, t_0, s_1, t_1 \in \mathbb{S}$ with $s_0 \equiv s_1$, $t_0 \equiv t_1$ and $|s_0| = |s_1| < |t_0| = |t_1|$.
 Let $r_i$ be the result of writing $s_i$ over $t_i$. Then $r_0 \equiv r_1$.
 \item Let $t \in \mathbb{S}$, let $\beta < |t|$ and let $r$ be the result of translating $t$ by $n$ beyond $\beta+1$.
 Then $r \equiv t$.
 \item For each $s \in \Sbb$ and countable $\beta \geq |s|$, there is a $t \in \Sbb$ such that $s \subseteq t$ and $|t| = \beta+1$.
\end{enumerate} 
\end{lemma} 

\begin{proof}
Routine.
\end{proof}

\begin{lemma}  \label{closure_properties} \leavevmode
\label{closure}
The following are true:
\begin{enumerate}
\item \label{S_fusion}
If $\seq{ s_n:n<\omega }$ is a sequence in $\mathbb{S}$ with $s_{n}\subseteq_n s_{n+1}$ then $(\bigcup_{n<\omega} s_n) \cat \seq{i} \in \mathbb{S}$ for all $i < \omega$.
\item \label{find_n-extensions}
If $s, t \in \mathbb{S}$ with $s\equiv t\restr \alpha$, then for any $n<\omega$ there is an $s\subseteq_n r$ such that $r\equiv t$. 
\item \label{P_sigma-closed}
Any decreasing sequence $\seq{ p_n:n<\omega }$ in $\mathbb{P}$ has a lower bound.
\end{enumerate}
\end{lemma}
\begin{proof} 
For (\ref{S_fusion}), let $s = \left( \bigcup_n s_n \right) \cat \seq{i}$.
Clearly $|s|$ is a countable successor ordinal, so we need only verify that $s$ is finite-to-one. 
For  $k < n < \omega$,
$$s^{-1}( \{ k \}) \subseteq s_n^{-1}( \{ k \}) \cup \{ \vert s \vert - 1 \}$$
 by the choice of the sequence  $\seq{ s_n : n<\omega  }$,
and this set is finite because $s_n \in \mathbb{S}$. 
To see (\ref{find_n-extensions}), let $r'$ be the result of writing $s$ over $t$, and let $r$ be the result of translating $r'$ by $n$ beyond $\alpha$.
By definition $s \subseteq_n r$, and by Lemma \ref{operations_and_ER} $r \equiv t$.
For (\ref{P_sigma-closed}), observe that by (\ref{find_n-extensions}), 
we may recursively choose $s_n \in \mathbb{S}$ such that $s_{n} \subseteq_n s_{n+1}$ and $p_n = [s_n]$.
By (\ref{S_fusion}), $[(\bigcup_{n < \omega} s_n) \cat \seq{0}]$ is a lower bound for $\seq{ p_n:n<\omega }$.
\end{proof}

Using Lemma \ref{closure_properties} it is straightforward to build $\leq_{\mathbb{P}}$-decreasing sequences $\seq{ p_\alpha:\alpha<\omega_1 }$ in $\mathbb{P}$ such that  $p_\alpha$ is the $\equiv$-class of some $t_\alpha:\alpha + 1 \rightarrow\omega$.
Given such a sequence, define 
\[
T:=\{t\in \mathbb{S}: t\in p_\alpha\text{ for some }\alpha<\omega_1\}=\bigcup_{\alpha<\omega_1} p_\alpha.
\]
Clearly $T$ is $\varrho_2$-coherent and full and hence an Aronszajn tree
by remarks made after Definitions \ref{Sbb:def} and \ref{rho2_coherent:def}.

Our general strategy to prove Theorem~\ref{baum_sol} now comes into focus.  What we need to do is to use $\diamondsuit$ to build a sequence $\seq{ t_\alpha:\alpha<\omega_1 }$ of elements of $\mathbb{S}$ such that:
\begin{itemize}
\item $t_\alpha:\alpha + 1 \rightarrow \omega$,
\item the sequence $\seq{ [t_\alpha]:\alpha<\omega_1 }$ is $\leq_{\mathbb{P}}$-decreasing in $\mathbb{P}$, and
\item the associated tree has the property that any subtree contains a frozen cone.
\end{itemize}
If we can do this, then Theorem~\ref{baum_sol} follows.

\begin{proof}[Proof of Theorem~\ref{baum_sol}]

Let $\seq{ A_\alpha:\alpha<\omega_1 }$ be a $\diamondsuit$-sequence, which we will assume is tailored to guess initial segments of $\omega_1$-trees from $\mathbb{S}$---i.e. if $S \subseteq \Sbb$ is an $\omega_1$-tree, then there are stationarily many limit ordinals $\delta<\omega_1$ with
$A_\delta = S_{<\delta}$,
the initial segment of $S$ of all levels prior to level $\delta$ (see discussion in the proof of Theorem \ref{construct} below).

Part of the construction is trivial:  we let $t_0=\seq{ 0  }$, and if we are given $t_\alpha$ then we set
$t_{\alpha+1}:=t_\alpha\vphantom{B}\!^\smallfrown\seq{ 0 }$.
Thus, the interesting case occurs when $\delta<\omega_1$ is a limit ordinal and we have constructed $\seq{ t_\alpha:\alpha<\delta }$. Under these circumstances, we will know what $T$ looks like below level $\delta$, and our choice of $t_\delta : \delta+1 \to \omega$ will determine which branches through this initial segment of $T$
will have continuations at level $\delta$.

The $\diamondsuit$-sequence presents us with a countable subtree $A_\delta$ of $\mathbb{S}$,
and we ask if $A_\delta$ is a subtree of $T_{<\delta}$ that does not contain a frozen cone
of $T_{<\delta}$.
If the answer to this question is ``no,''
then we need not worry about $A_\delta$ and let $t_\delta : \delta + 1 \to \omega$ be any element of $\Sbb$
such that $[t_\delta]$ is a lower bound of
$\seq{ [t_\alpha] :\alpha<\delta }$ in $\Pbb$.
If the answer is ``yes,'' then we will need to choose $t_\delta : \delta +1 \to \omega$ in $\Sbb$
so that for each $s \equiv t_\delta$, there is $\alpha < \delta$ such that
$s \restriction \alpha+1$ is not in $A_\delta$.

We do this in countably many steps.
First, we let $\seq{ \delta_n:n<\omega }$ be an increasing sequence cofinal in $\delta$.
In our construction, we will be choosing ordinals $\alpha_n$ and
corresponding $s_n\in p_{\alpha_n}$ such that: 
\begin{itemize}
    \item $\delta_n\leq\alpha_n<\delta$, and
    \item $s_{n}\subseteq_n s_{n+1}$.
\end{itemize}
This guarantees that
\[
    t_\delta:=\left( \bigcup_{n<\omega}s_n \right) \cat \seq{ 0  }
\]
will be in $\mathbb{S}$ and of the right length.

We start with $\alpha_0=\delta_0$, and let $s_0 = t_{\alpha_0}$.
Once we have constructed $s_n$ and $\alpha_n$, we assume that some bookkeeping process hands us a $\varrho_2$-modifier $m_n:\delta+1\rightarrow\mathbb{Z}$.
The function $m_n$ should be thought of as coding a member of the equivalence class of the $t_\delta$ we are building.
Thus, we look at the function $s_n+m_n$ defined on $|s_n|$ given by
\[
   (s_n+m_n)(\xi):=s_n(\xi)+m_n(\xi)
\]
and ask if this is a member of $A_\delta$.
If the answer is ``no''
(this includes the case in which $s_n + m_n$ has negative values), 
    then we choose $\alpha_{n+1}$ to be greater than $\delta_{n+1}$, and let $s_{n+1}$ be some $n$-extension of $s_n$ in 
    the $\equiv$-equivalence class $[t_{\alpha_{n+1}}]$. 
If the answer is ``yes,'' then we finally need to use our assumption that $A_\delta$,
when considered as a subtree of $T_{<\delta}$, does not contain a frozen cone of $T_{<\delta}$.

Define
\[
    M:=\max\{|m_n(\xi)|:\xi<\delta\}
\qquad \textrm{ and } \qquad
    N:=M+n+1.
\]
Our assumption says that $s_n+m_n$ will have an $N$-extension $r$ in $T_{<\delta}$ that is not in $A_\delta$.
Extending $r$ will not change this situation, so we may assume that $|r| = \alpha_{n+1} +1$ where
\[
    \alpha_{n+1} \geq \delta_{n+1}.
\]
Now the idea is that we should define
\[
    s_{n+1}:=r-m_n.
\]
Notice that if $\xi<\alpha_n$, then
\[
    s_{n+1}(\xi)=r(\xi)-m_n(\xi)=s_n(\xi)+m_n(\xi)-m_n(\xi)=s_n(\xi),
\]
and so $s_{n+1}$ extends $s_n$.
If $\alpha_n\leq\xi<\alpha_{n+1}$, then
$r(\xi)\geq N$
and hence 
\[
    s_{n+1}(\xi) =r(\xi)-m_n(\xi)\geq n+1.
\]
Thus $s_{n+1}$ is in fact an $n$-extension of $s_n$.
Finally, $s_{n+1}\equiv t_{\alpha_{n+1}}$ because $s_{n+1}$ is equivalent to $r$ and $r\in T_{<\delta}$. 
The key point is that if $t$ is any extension of $s_{n+1}$ in $T_\delta$, then applying the modification $m_n$ to $t$ results in some $s$ such that $s \restriction \alpha_n + 1$ is not in $A_\delta$. 

Since we made sure to arrange $s_n \subseteq_n s_{n+1}$, we know 
\[
t_\delta:=\Big(\bigcup_{n<\omega}s_n \Big) \cat \seq{0},
\]
is in $\mathbb{S}$ and of height $\delta$.
Thus, $[t_\delta]$ will be a lower bound for $\seq{ [t_\alpha] : \alpha<\delta }$ in $\mathbb{P}$,
and the $\r2$-modifications of $t_\delta$ will be the $\delta\Th$ level of $T$. 

The construction described above will produce a decreasing sequence
$\seq{ [t_\alpha]:\alpha<\omega_1 }$ in $\mathbb{P}$.
It remains to show that every subtree of $T$ contains a frozen cone.
Suppose $S \subseteq T$ is downward closed and does not contain a frozen cone.
By the choice of our $\diamondsuit$-sequence, there must be a $\delta<\omega_1$ such that
$A_\delta = S_{<\delta}$ and $(T_{<\delta}, <_T, S_{<\delta}) \prec (T, <_T, S)$.
In particular $\delta$ is a limit ordinal, $A_\delta \subseteq T_{<\delta}$, and $A_\delta = S_{<\delta}$ contains no frozen cone of $T_{<\delta}$.

It suffices to show that $S\subseteq T_{<\delta}$.
This follows from our construction, though:
if $t$ is any element of level $\delta$ of $T$,
then during our construction of $t_\delta$ there was a stage where the function $t-t_\delta$ appeared as $m_n$.
Since $S$ does not contain a frozen cone, $s_{n+1}$ was chosen so that $s_{n+1} + m_n$ is in $T_{<\delta} \setminus S_{<\delta}$.
Because $t$ extends $s_{n+1} + m_n$, $t$ is also not in $S$.
Thus, the height of $S$ is at most $\delta$ and so $S$ is countable.  
We conclude that any subtree of $T$ contains a frozen cone, as required.
\end{proof}

\section{Countryman lines at higher cardinals}

\label{higher_coherence}

In the remainder of the paper, our aim is to adapt the construction in the previous section to higher cardinals.
While this is of interest in its own right, our main motivation is to produce examples of minimal non-$\sigma$-scattered linear
orders of cardinality $\kappa^+ > \aleph_1$.
In fact these orders will be $\kappa^+$-Countryman lines.

We will begin recording some some basic facts about \emph{$\kappa^+$-Countryman lines},
when $\kappa$ is an infinite cardinal.
Here a linear order $C$ is $\kappa^+$-Countryman if its cardinality is $\kappa^+$ and $C^2$ is
the union of $\kappa$ chains with respect to the coordinatewise partial order on $C^2$.

\begin{lemma} \label{non-Countryman}
    Suppose that $L$ is a linear order of cardinality $\kappa^+$ and that whenever $Z \subseteq L \times L$
    is a chain, there are at most $\kappa$ elements $x \in L$ such that
    \[
Z_x := \{y \in L : (x,y) \in Z\}
    \]
    has cardinality $\kappa^+$.
    Then $L$ is not $\kappa^+$-Countryman.
\end{lemma}

\begin{proof}
    Suppose that $\Zscr$ is a collection of chains in $L \times L$ with $|\Zscr| = \kappa$.
    Since $\kappa^+$ is not a union of $\kappa$ sets of cardinality $\kappa$, our assumption implies
    there is an $x \in L$ such that for every $Z \in \Zscr$,
    $Z_x$ has cardinality at most $\kappa$.
    Again using the regularity of $\kappa^+$,
    there is a $y \in L$ such that $y \not \in Z_x$ for every
    $Z \in \Zscr$.
    But now $(x,y) \in L \times L$ is not covered by $\Zscr$.
    Since $\Zscr$ was arbitrary, $L$ is not $\kappa^+$-Countryman.
\end{proof}

\begin{proposition}
Suppose that $C$ is $\kappa^+$-Countryman.
The following are true:
\begin{enumerate}
    \item \label{conv+hered}
    $C^*$ is $\kappa^+$-Countryman and any suborder of $C$ of cardinality $\kappa^+$ is $\kappa^+$-Countryman.
    \item \label{notWO}
    $C$ is not a well order.
    \item \label{large_density}
    $C$ is has no dense suborder of cardinality $\kappa$.
    \item \label{C-line->A-line}
    $C$ is $\kappa^+$-Aronszajn.
    \item \label{not_near}
    If $L$ is a linear order which embeds into $C$ and $C^*$, $|L| \leq \kappa$.
    \end{enumerate}
\end{proposition}

\begin{proof}
Item (\ref{conv+hered}) is trivial and (\ref{C-line->A-line}) is an immediate consequence of
(\ref{conv+hered})--(\ref{large_density}).
To see (\ref{notWO}), observe that by (\ref{conv+hered}), it suffices to show that $\kappa^+$ is not
Countryman.
Notice that if $Z \subseteq \kappa^+ \times \kappa^+$ is a chain and some section $Z_\alpha$ has cardinality $\kappa^+$,
then it is cofinal in $\kappa^+$ and hence $Z_{\alpha'}$ is empty whenever $\alpha < \alpha'$.
In particular, there is at most one $\alpha$ such that $Z_\alpha$ has cardinality $\kappa^+$.
By Lemma \ref{non-Countryman} $\kappa^+$ is not Countryman.

To see (\ref{large_density}), suppose that $C$ has cardinality $\kappa^+$ and yet has a dense subset $D$ of cardinality
$\kappa$.
If $Z \subseteq C \times C$ is a chain, let $X$ be the set of all $x \in C$ such that the section
$Z_x$ contains at least two elements $a_x < b_x$.
As $D$ is dense, whenever $x \in X$ we may choose $d_x \in D$ with $a_x < d_x < b_x$.
Since $x < x'$ implies $Z_x < Z_{x'}$, it also implies $d_x \ne d_{x'}$.
Thus $|X| \leq |D| \leq \kappa$.
Again, by Lemma \ref{non-Countryman}, $C$ is not $\kappa$-Countryman.

Finally, to see (\ref{not_near}),
notice that if $L$ is any linear order and $f:L \to C$ and $g:L \to C^*$ are order preserving, then 
$\{(f(x),g(x)) : x \in L\}$ meets any chain in $C^2$ in at most one point.
In particular, if $C$ is $\kappa^+$-Countryman, $|L| \leq \kappa$.
\end{proof}

Notice that the definitions of \emph{$\varrho_2$-modification} and \emph{$\varrho_2$-coherent}
which we made previously makes
sense in the generality of $\functions{<\kappa^+}{\omega}$.
A subset $X$ of $\functions{<\kappa^+}{\omega}$ is \emph{$\varrho_2$-full with respect to $\kappa^+$} if 
it has cardinality $\kappa^+$ and is closed under initial segments of successor length and $\varrho_2$-modifications.
If $\kappa^+$ is clear from the context, we will sometimes abuse notation and write ``$\varrho_2$-full'' (or just ``full'')
to mean ``$\varrho_2$-full with respect to $\kappa^+$.''
The next proposition provides a useful criterion for demonstrating that a linear order is $\kappa^+$-Countryman.
The proof is a routine modification of arguments of Todorcevic \cite{acta} and is included for completeness.
Recall that a tree of height $\kappa^+$ is \emph{special} if it is a union of $\kappa$ antichains.

\begin{proposition} \label{kappa-Countryman}
Suppose that $T \subseteq \functions{<\kappa^+}{\omega}$ is $\varrho_2$-coherent and has cardinality $\kappa^+$.
If $T$ is special, then $(T,\leq_\lex)$ is $\kappa$-Countryman.    
\end{proposition}

\begin{proof}
It suffices to cover $\{(s,t) \in T^2 : |s| \leq |t|\}$ 
by $\kappa$ many chains.
Given $(s,t) \in T^2$ with $|s| \leq |t|$,
let $n = n(s,t)$ and $\xi_i = \xi_i(s,t)$ for $i \leq n$
be such that:
\begin{itemize}

\item $\xi_0 = 0 < \xi_1 < \ldots < \xi_n = |s|$,

\item $t(\xi_i) - s(\xi_i) \ne t(\xi_{i+1}) - s(\xi_{i+1})$,
and

\item if $\xi_i < \eta < \xi_{i+1}$, then
$t(\eta) - s(\eta) = t(\xi_i) - s(\xi_i)$.

\end{itemize}
Let $f:T \to \kappa$ be such that $f^{-1}(\alpha)$ is an
antichain for each $\alpha < \kappa$.
Define $\sigma(s,t)$ and $\phi(s,t)$ to be the sequences of length $n(s,t)$
given by
$$\sigma(s,t)(i) := t(\xi_i)-s(\xi_i) \qquad \qquad
\phi(s,t)(i) := f(s \restriction \xi_{i+1})$$
whenever $i < n(s,t)$.

Since the sets of possible values of $\sigma$ and $\phi$ have cardinality $\kappa$,
it suffices to show that if $\sigma(s,t) = \sigma(s',t')$
and $\phi(s,t) = \phi(s',t')$, then either:
\begin{itemize}

\item  
$s \leq_\lex s'$ and $t \leq_\lex t'$ or

\item 
$s' \leq_\lex s$ and $t' \leq_\lex t$.

\end{itemize}
Notice that this is vacuously true if either $s = s'$ or $t=t'$.
For ease of reading, we will write $\xi_i$ for $\xi_i(s,t)$
and $\xi_i'$ for $\xi_i(s',t')$.
Let $i \leq n$ be maximal such that $\xi_i = \xi_i'$.
If $i=n$ and $s=s'$, then the desired conclusion follows.
Otherwise set $\zeta = |s|$ if $i = n$ and
$\zeta = \min(\xi_{i+1},\xi_{i+1}')$ if $i < n$.

\begin{claim}
$s \restriction \zeta \neq s' \restriction \zeta$.
\end{claim}

\begin{proof}
If $i = n$ then $\xi_n = \xi'_n = \vert s \vert = \vert s' \vert = \zeta$,
and we are done since $s \neq s'$.
Thus we may assume that $i < n$.
Since $s \restriction \xi_{i+1} \ne s' \restriction \xi_{i+1}'$ and
$f(s \restriction \xi_{i+1}) = f(s' \restriction \xi_{i+1}')$, it follows
that $s \restriction \xi_{i+1}$ is incompatible with $s' \restriction \xi_{i+1}'$ and
therefore that $s \restriction \zeta \ne s' \restriction \zeta$.
\end{proof}

By exchanging the roles of $s$ and $s'$ if necessary
assume that $s <_\lex s'$.
Observe that since
$\sigma(s,t) = \sigma(s',t')$, 
$$
t(\eta) - s(\eta) = t'(\eta) - s'(\eta)
$$
and hence
\begin{equation}
\label{same_difference}
t(\eta) - t'(\eta) = s(\eta) - s'(\eta)
\end{equation}
whenever $\eta < \zeta$.
Let $\delta$ be minimal such that
$s(\delta) \ne s'(\delta)$.
Since $\delta < \zeta$, (\ref{same_difference}) implies
$t \restriction \delta = t' \restriction \delta$ and
$t(\delta) < t'(\delta)$.
Thus $t <_\lex t'$, as desired.
\end{proof}

\begin{proposition}
Suppose that $C \subseteq \functions{<\kappa^+}{\omega}$ is $\varrho_2$-coherent and full.
If $X \subseteq C$ has cardinality at most $\kappa$, then $(X,\leq_\lex)$ is $\sigma$-scattered. 
\end{proposition}

\begin{proof}
    By adding 1 to all of the values of elements of $X$ if necessary, we may assume that no element of 
    $X$ takes the value $0$.
    Let $t \in C$ be such that $|t|$ is an upper bound for the lengths of elements of $X$, and let $Y$ be the set of all
    $\varrho_2$-modifications of $t$.
    Define $f: X \to Y$ by
    $$
    f(s)(\xi):=
    \begin{cases}
    s(\xi) & \textrm{ if } \xi < |s| \\
    0 & \textrm{ if } \xi = |s| \\
    t(\xi) & \textrm{ if } \xi > |s| \\
    \end{cases}
    $$
    and observe that $f$ preserves $\leq_\lex$ (since we've arranged $s$ only takes positive values,
    $0$ effectively serves as a terminating symbol for the sequence and we've defined
    $\leq_\lex$ so that the terminating symbol is less than all other symbols).
Since $y \mapsto y-t$ also preserves $\leq_\lex$ and maps $Y$ into the set of $\varrho_2$-modifiers of length $\vert t \vert$,
we are done by Proposition \ref{mod_sigma-scat}.
\end{proof}

\section{Higher \texorpdfstring{$\r2$}{r2}-coherence and the tree \texorpdfstring{$\Sbb_\kappa$}{Sk}}

\label{S_kappa:sec}

In order to apply Proposition \ref{kappa-Countryman}, it will be helpful to construct an analog $\Sbb_\kappa \subseteq \functions{<\kappa^+}{\omega}$
of $\Sbb$ for higher cardinals $\kappa$ such that any $T \subseteq \Sbb_{\kappa}$ which is $\varrho_2$-coherent and full is special.
Toward this end, let us assume that $\kappa$ is a (possibly singular) infinite cardinal.
If there is a $\Box_\kappa$-sequence, then the tree $T(\varrho_2)$ defined using minimal walks down the $\Box_\kappa$-sequence has many nice coherence properties.
Our plan is to capture some of this structure in an abstract way.

\begin{definition} 
\label{r2like}
Define $\Sbb_\kappa$ to consist of all functions $t \in \functions{<\kappa^+}{\omega}$ which satisfy the following conditions: 
\begin{enumerate}
    \item \label{cdn1a} $|t|=\delta+1$ for some $\delta<\kappa^+$ (which we denote as $\Top(t)$),
    \item \label{cdn1b} for each integer $n \geq -1$, the set 
    $
        C^t_n:=\{\alpha < |t|: t(\alpha) \leq n\}
    $
    is closed,
    \item \label{cdn1c} if $\alpha < |t|$ is a limit ordinal, then $t(\alpha)$ is the least $n$
    such that $C^t_n$ is unbounded in $\alpha$ (noting that $C^t_{-1} = \emptyset$), and 
    
    \item \label{cdn1d}if $I$ is a maximal open interval of $|t|$ that is disjoint to $C^t_{n-1}$ then
  $$
      \otp(C^t_{n}\cap I)<\kappa\cdot \omega.
  $$
   \end{enumerate}
If $t \in \Sbb_\kappa$, then we let $\last(t)$ (the last value of $t$)  be given by
$$
    \last(t):=t(\Top(t)).
$$
\end{definition}

Setting $n = 0$ and $I = \vert t \vert$ in (\ref{cdn1d}), $\otp(C^t_0) < \kappa \cdot \omega$. An easy induction (break up $\vert t \vert$ into blocks demarcated by elements
of the closed set $C^t_n$) now shows that $\otp(C^t_n) < (\kappa \cdot \omega)^{n+1}$ for all $n < \omega$.

Observe that if $s \ne t$ are in $\Sbb_\kappa$, $\last(s) = \last(t)=:n$, and $s \subseteq t$, then $C^s_n$ is a proper initial segment of
$t$ and hence $\otp(C^s_n) < \otp(C^t_n)$.
In particular,
$$s \mapsto (\last(s),\otp(C^s_{\last(s)}))$$is a specializing function for $\Sbb_\kappa$.
(The use of the specific ordinal $\kappa\cdot \omega$ in the definition is not critical: $\kappa\cdot \omega$ is large enough to guarantee that $\mathbb{S}_\kappa$ will be closed under certain types of increasing unions, but small enough
to ensure that our specializing function takes values in a set of cardinality $\kappa$.)

The definition of $\subseteq_n$ given in Section \ref{baumgartner:sec} generalizes without change to $\Sbb_\kappa$,
as does the definition of \emph{frozen cone}.
The following proposition summarizes what we have shown so far; the proof of the later statement is obtained from the arguments in
Section \ref{baumgartner:sec} \emph{mutatis mutandis}.

\begin{proposition}
If $T \subseteq \Sbb_\kappa$ is $\varrho_2$-coherent and full,
then $(T,\leq_\lex)$ is a $\kappa^+$-Countryman line and any suborder of cardinality at most $\kappa$ is $\sigma$-scattered.
Moreover, if every subtree of $T$ contains a frozen cone, then $(C,\leq_\lex)$ is a minimal non-$\sigma$-scattered linear order,
whenever $C \subseteq T$ is an antichain of cardinality $\kappa^+$.
\end{proposition}

Unlike in Section \ref{baumgartner:sec}, it need not be the case that in a given model of set theory that there is
a subset $T$ of $\Sbb_\kappa$ which is $\varrho_2$-coherent and full when $\kappa > \aleph_0$---after all such a $T$ is a $\kappa^+$-Aronszajn tree and hence
witnesses the failure of the \emph{tree property} at $\kappa^+$ (see \cite{TP}).
On the other hand, if $\varrho_2$ is defined from a $\Box_\kappa$-sequence as in \cite{acta}, 
then the collection of all $\varrho_2$-modifications of
$$\{\varrho_2(\cdot ,\beta) \restriction \alpha+1 : \alpha < \beta < \kappa^+\}$$
is a subset of $\Sbb_\kappa$ which is $\varrho_2$-coherent and full \cite{acta}.

We will now establish some basic properties of $\Sbb_\kappa$ and define some terminology and notation.

\begin{lemma}
 \label{r2structure}
 Suppose $t\in\mathbb{S}_\kappa$ and $\delta=\Top(t)$. 
 \begin{enumerate} 
    \item \label{increading_closed} The sequence $\seq{ C^t_n:n<\omega }$ is an increasing sequence of closed sets with union $\delta+1=|t|.$

    \item \label{locally_constant} If $\alpha < |t|$ is a limit ordinal, then $t(\alpha)=n$ implies that $t$ is constant with value $n$ on a closed unbounded subset of $\alpha$. 

    \item \label{nonacc_succ} For each $n<\omega$ the set $\nacc(C^t_n)$ of non-accumulation points of $C^t_n$ consists of successor ordinals.

    \item \label{limits_determined} The function $t$ is determined by its values on successor ordinals.

\end{enumerate}
\end{lemma}

\begin{proof}
Item (\ref{increading_closed}) is immediate from the definitions.
For (\ref{locally_constant}), assume that $t(\alpha)=n$.
Both $C^t_{n-1}\cap \alpha$ and $C^t_n \cap \alpha$ are closed in $\alpha$, but the former is bounded below $\alpha$ while the latter is not.
Thus $C^t_n \cap \alpha\setminus \sup(C^t_{n-1})$ is closed and unbounded in $\alpha$.
But since this is contained in the set of $\beta<\alpha$ for which $t(\beta)=n$, we are done.   
Statements (\ref{nonacc_succ}) and (\ref{limits_determined}) now follow immediately. 
\end{proof}

The collection $\mathbb{S}_\kappa$ is closed under some natural operations.
For example, it is clear that this set is closed under restrictions to successor ordinals.
Also if $t\in \mathbb{S}_\kappa$, then so is $t^\smallfrown\seq{ n }$ for every $n<\omega$. 
Most important for us, though, is that $\mathbb{S}_\kappa$ is essentially closed under certain types of increasing unions. The next definition will help us analyze the situation.

\begin{definition}
Given a limit ordinal $\delta<\kappa^+$, a function $t:\delta\rightarrow\omega$ is an \emph{$\Sbb_\kappa$-limit}
if $t\restr\alpha+1$ is in $\Sbb_\kappa$ for every $\alpha<\delta$.
\end{definition}

The point is that any strictly $\subseteq$-increasing union of elements
of $\Sbb_\kappa$ is an $\Sbb_\kappa$-limit. 
If $t$ is an $\Sbb_\kappa$-limit with domain some limit ordinal $\delta$, then $t$ will possess many of the characteristics of an element of $\Sbb_\kappa$ automatically.
For example, the definition of $C^t_n$ makes sense for each $n$, and these sets will each be closed in $\delta$ because all of their proper initial segments are closed.
We also note that if $t$ does have an extension $s \in \Sbb_\kappa$ with $\Top(s)=\delta$, then in fact this extension is {\em unique}, because the value $s(\delta)$ must be the least $n$ for which $C^t_{n+1}$ is unbounded in $\delta$.
We will encounter this idea many times, so it will be convenient to give this particular $n$ a name.

\begin{definition}
Suppose $t:\delta\rightarrow\omega$ for some limit ordinal $\delta<\kappa^+$. 
The limit infimum of $t$, denoted $\liminf(t)$ is defined to be the least $n<\omega$ with pre-image unbounded in $\delta$ if such an $n$ exists, and is said to be $\infty$ otherwise.
\end{definition}

Notice that if $\delta$ has uncountable cofinality, then any $t:\delta\rightarrow\omega$ will have a finite limit
infimum by a simple counting argument.
Thus, the situation $\liminf(t)=\infty$ is possible only if $\cf(\delta)=\omega$.  

For an $\Sbb_\kappa$-limit $t$,
the question of whether $t$ can be extended to an element of $\Sbb_\kappa$ hinges on the existence of a finite limit
infimum whose pre-image is not too large.
The following lemma makes this precise.

\begin{lemma}
\label{r2extension}
Suppose $t$ is an $\Sbb_\kappa$-limit with domain some limit ordinal $\delta<\kappa^+$.
Then the following two statements are equivalent:
\begin{itemize}

\item $t$ has an extension $s \in \Sbb_\kappa$ with $\Top(s)=\delta$.

\item  $\liminf(t)$ is some finite $n<\omega$, and  the pre-image of $n$ under $t$ has a tail of order-type less than $\kappa\cdot\omega$.

\end{itemize}
In particular, if $t$ is an $\Sbb_\kappa$-limit and $|t|=\delta$ has uncountable cofinality,
then $t$ has an extension $s \in \Sbb_\kappa$ with $\Top(s)=\delta$.
\end{lemma}
\begin{proof}
For the forward implication, suppose $s \in \Sbb_\kappa$ is an extension of $t$ with $\Top (s) = \delta$.
Since $s \in \Sbb_\kappa$, $\liminf(t) = s(\delta)$ is finite.
If $\alpha < \delta$ is such that $t \ge s(\delta)$ on the interval $(\alpha, \delta]$, 
then $$\otp( \{ \eta \in (\alpha, \delta) : t(\eta) = s(\delta) \} ) < \kappa \cdot \omega$$
because $s$ satisfies requirement (\ref{cdn1d}) in Definition \ref{r2like}.

For the reverse implication assume $t$ satisfies $\liminf(t)=n$. 
We want to show that the function
$s:=t^\smallfrown\seq{ n }$
is in $\Sbb_\kappa$.
Since $t$ is an $\Sbb_\kappa$-limit and $s(\delta)=\liminf(t)=n$,
requirements (\ref{cdn1a})--(\ref{cdn1c}) of Definition~\ref{r2like} are easily satisfied.

For requirement (\ref{cdn1d}), let $m$ be given and $I \subseteq \delta +1$ be an open interval on which $s > m$.
If $m \geq n$, then since $s^{-1}(n)$ is cofinal in $\delta$, it must be that $\beta := \sup(I) < \delta$.
Since $t \restriction \beta+1$ is in $\Sbb_\kappa$, it follows
that $$\otp(C^s_{m+1} \cap I) = \otp(C^{t \restriction \beta+1}_{m+1} \cap I) < \kappa \cdot \omega.$$
If $m < n-1$, 
then $s^{-1}(m+1) = t^{-1}(m+1)$ is bounded by some $\beta < \delta$ and we are again done by virtue of $t \restriction \beta+1$ being
in $\Sbb_\kappa$.
Finally, if $m = n-1$, then by our hypothesis we may write $I \cap \delta = I_0 \cup I_1$,
where $I_0$ is an initial segment of $I \cap \delta$ which is bounded in $\delta$,
and $I_1$ is a tail of $I \cap \delta$ such that $\otp(\{\eta \in I_1: t(\eta) = n\}) < \kappa \cdot \omega$.
Since $t$ is an $\Sbb_\kappa$-limit,
we have $\otp(\{\eta \in I_0: t(\eta) = n\}) < \kappa \cdot \omega$.
Since $\kappa \cdot \omega$ is closed under ordinal addition, it follows that 
$\otp(\{\eta \in I: s(\eta) = n\}) < \kappa \cdot \omega$ as required.
\end{proof}

This simplifies the project of building $\subseteq$-increasing sequences in $\mathbb{S}_\kappa$ immensely:
we just need to worry about what happens at limit stages of countable cofinality.
In particular, we need to guarantee that the limit infimum is finite
and that the order-type of its pre-image does not grow to ordertype $\kappa\cdot\omega$.
This turns out to be relatively easy to arrange provided we are careful at successor stages.
The following definition formulates a straightforward way of doing this.

\begin{definition}
\label{capextdef}
Suppose $s,t \in \mathbb{S}_\kappa$.
We say that $t$ is a {\em capped} extension of $s$ if:
\begin{itemize}

    \item $s\subset t$ (so $t$ properly extends $s$),

    \item $\last(t)=0$ (so $t$ terminates with the value $0$), and

    \item $t(\xi)>0$ for all $|s| \leq \xi < |t| -1$ (so $\Top(t)$ is the only place beyond $s$ where $t$ returns the value $0$).

\end{itemize}
\end{definition}

The motivation for this definition is as follows.
Suppose that $\seq{ s_n:n<\omega }$ is a sequence of elements of $\mathbb{S}_\kappa$
and that $s_{n+1}$ is a capped extension of $s_n$ for all $n<\omega$.
The definition guarantees that the union $t$ of the chain will satisfy $\liminf(t)=0$, and 
\[
\otp(t^{-1}(\{0\})) = \otp(s_0^{-1}(\{0\}))+\omega<\kappa\cdot\omega.
\]
Thus, the sequence $\seq{ s_n:n<\omega }$ can be continued in a canonical way: we can define
\[
s_\omega :=t \cat \seq{ 0 }.
\]
The function $s_\omega$ so defined is in fact a {\em least} upper bound for the sequence in $\mathbb{S}_\kappa$,
as any such extension must take on the value $0$ at $\Top(s)$. 
We now extend this notion to longer sequences in the obvious way.

\begin{definition}
\label{capseqdef}
A $\subset$-increasing sequence $\bar{s}=\seq{ s_\beta:\beta<\alpha }$ of elements of $\Sbb_\kappa$ is {\em capped} if:
\begin{itemize}
    \item $s_{\beta+1}$ is a capped extension of $s_\beta$ for all $\beta<\alpha$.

    \item for $\gamma<\alpha$ a limit ordinal, we have
    \[
        s_\gamma=\Big(\bigcup_{\beta<\gamma} s_\beta\Big)\vphantom{B}^\smallfrown \seq{ 0 }.
        \]
    (So for limit $\gamma$, $s_\gamma$ is the canonical extension of the sequence $\seq{ s_\beta:\beta<\gamma }$ in $\mathbb{S}_\kappa$.)
    \end{itemize}
\end{definition}

We now have all the pieces we need to easily get our sufficient condition for building $\subseteq$-increasing sequences in $\mathbb{S}_\kappa$ that are guaranteed to have upper bounds.

\begin{lemma}
\label{closurelemma}
A capped sequence in $\mathbb{S}_\kappa$ of length at most~$\kappa$ has a least upper bound in $(\mathbb{S}_\kappa,\subseteq)$.
\end{lemma}
\begin{proof}
Let $\seq{s_\alpha : \alpha < \gamma}$ be a capped sequence in $\Sbb_\kappa$ for $\gamma \leq \kappa$ and let $t = \bigcup_{\alpha < \gamma} s_\alpha$.
Let $\delta_\alpha = |s_\alpha|$ and observe that $\{\delta_\alpha : \alpha < \gamma\}$ is a closed unbounded set in $|t|$ and moreover
is a tail of $t^{-1}(0)$.
In particular, $\liminf t = 0$ and a tail of $t^{-1}(0)$ has ordertype $\gamma \leq \kappa < \kappa \cdot \omega$.
By Lemma \ref{r2extension}, $t \cat \seq{0}$ is in $\Sbb_\kappa$.
\end{proof}

Next we turn to modifications of elements of $\Sbb_\kappa$.
For the sake of simplicity, in the context of $\Sbb_\kappa$, all modifiers will have length $\kappa^+ + 1$
and will be identified with their restriction to $\kappa^+$ as per our remark in Section \ref{baumgartner:sec}.

It will be helpful to define some notation and terminology associated to a given modifier.

\begin{definition}
Suppose $m:\kappa^+\rightarrow\mathbb{Z}$ is a modifier.
\begin{enumerate}
\item The {\em height of $m$}, denoted $\h(m)$, is the least $\zeta<\kappa^+$ for which $h$ is constant on $[\zeta,\kappa^+)$.

\item The {\em norm of $m$}, denoted $\|m\|$, is the the maximum value of the form $|m(\xi)|$.
Equivalently $\|m\|$ is the least
$N$ such that every value of $m$ is in $[-N,N]$.

\item  We define the {\em change set of $m$}, denoted $\Delta(m)$,  to consist of the ordinals
$\xi_0 = 0 < \xi_1 < \cdots \xi_n = \h(m)$ such that $m$ is constant on $[\xi_i,\xi_{i+1})$ for each $i < n$
and $m(\xi_{i+1}) \ne m(\xi_i)$;
the ordinals $\xi_i$ for $0 < i \leq n$ are the {\em change points} of $m$.

\end{enumerate}
\end{definition}

We say that a modifier $t$ is \emph{legal} for $s \in \Sbb_\kappa$ if the values of $s+t$ are nonnegative.
The motivation for this definition is the following lemma.

\begin{lemma} \label{legal_mod_Skappa}
If $s \in \Sbb_\kappa$ and $m$ is a $\varrho_2$-modifier which is legal for $s$, then $s+m \in \Sbb_\kappa$.
\end{lemma}
\begin{proof}
Clearly $|s+m| = |s|$ is a successor ordinal and $s+m$ takes values in $\omega$,
so requirement (\ref{cdn1a}) of Definition (\ref{r2like}) is satisfied.
To verify requirements (\ref{cdn1b}) and (\ref{cdn1c}), let $\alpha < |s|$ with $\alpha$ limit, and let $s(\alpha) = k$ and
$m(\alpha) = l$.
Since $m$ is continuous $m(\beta) = l$ for all large $\beta < \alpha$, and since $s \in \Sbb_\kappa$,
 $s(\beta) \ge k$ for all large $\beta < \alpha$ and $s(\beta) = k$ for cofinally many $\beta < \alpha$.
 It follows that $(s+m)(\alpha) = k + l$,  $(s+m)(\beta) \ge k + l$
 for all large $\beta < \alpha$, and $(s+m)(\beta) = k +l$ for cofinally
 many $\beta < \alpha$.

 As for requirement (\ref{cdn1d}), let $I$ be an interval such that $(s+m)(\alpha) \ge n$ for all $\alpha \in I$,
 and break up $I$ into finitely many disjoint subintervals $I_i$ for $i < k$
 such that $m$ is constant on $I_i$ with value $l_i$.
 For $\alpha \in I_i$ we have that $s(\alpha) = (s+m)(\alpha) - l_i \ge n - l_i$, so that 
$$\otp(\{ \alpha \in I_i : s(\alpha) = n - l_i \}) < \kappa \cdot \omega,$$ which implies that
 $\otp(\{ \alpha \in I_i : (s+m)(\alpha) = n \}) < \kappa \cdot \omega$.
 Since the ordinal $\kappa \cdot \omega$ is
 closed under finite sums,
 $\otp(\{ \alpha \in I : s(\alpha) = n  \}) < \kappa \cdot \omega$
 and we have verified requirement (\ref{cdn1d}). 
\end{proof}

Notice that it follows immediately from Lemma \ref{legal_mod_Skappa} that if $s \in \Sbb_\kappa$
and $m$ is legal for $s$, then $-m$ is legal for $s+m$, in which case $(s+m)-m = s$.

\begin{proposition}
Suppose $s:\delta+1\rightarrow \omega$ is in $\Sbb_\kappa$ and let $\alpha<\delta$.
\begin{enumerate}
    \item \label{modified_extension} 
    The sequence $s$ has a modification $t$ that extends $s\restr\alpha$ and satisfies $t(\delta)=0$.
    
    \item \label{capped_extension} If $\delta$ is a successor ordinal, then $s$ has a modification that is a capped extension of $s\restr\alpha$.
\end{enumerate}
\end{proposition}
\begin{proof}
Part (\ref{modified_extension}) is immediate except for the case where $\delta$ is a limit ordinal for which $n:=s(\delta)>0$.
Since $s\in\mathbb{S}_\kappa$, $n$ is the least element of $\omega$ whose pre-image is unbounded in $\delta$. 
Increasing $\alpha$ if necessary, we may assume that $s(\xi) \ge n$ for $\xi \ge \alpha$. 
Now we can define a function $m:\kappa^+\rightarrow \mathbb{Z}$ by 
\[
    m(\xi)=
    \begin{cases}
    0  &\text{if $\xi\leq\alpha$, and}\\
    -n  &\text{if $\alpha<\xi<\kappa^+$.}
    \end{cases}
\]
The function $m$ is a modifier, and by the choice of $\alpha$ we know that it is legal for $s$.  The function
$t:= s+m$
has all the required properties.

Now suppose $\delta = \gamma+1$. Part (\ref{capped_extension}) is easy if $\alpha = \gamma$, so let us assume $\alpha <\gamma$ and define a modifier $m:\kappa^+\rightarrow \omega$ by
\[
    m(\xi)=
    \begin{cases}
    \hphantom{-s}0 &\text{if $\xi\leq\alpha$,}\\
    \hphantom{-s}1 &\text{if $\alpha<\xi\leq\gamma$, and}\\
    -s(\delta) &\text{if $\xi=\delta$.}
    \end{cases}
\]
Now $m$ is legal for $s$, and $s+m$ is a capped extension of $s\restr\alpha$ that is equivalent to $s$.
\end{proof}

\section{Forcing an Example}

\label{forcing:sec}

Recall that in Section \ref{baumgartner:sec} we derived a poset $\mathbb{P}$ from the set $\mathbb{S}$,
investigated the properties of $\mathbb{P}$ as a forcing poset, and  then used
$\diamondsuit$ and $\mathbb{P}$ to construct a subtree of $\mathbb{S}$ which permitted us to answer Baumgartner's question.
The argument of Section \ref{baumgartner:sec} easily shows that forcing with $\mathbb P$
adds a suitable tree, and indeed we may view the $\diamondsuit$ construction as building an $\omega_1$-sequence 
of elements which generates a sufficiently generic filter.

By analogy with the definition of $\mathbb P$ from ${\mathbb S}$, we let $[s]$ be the set of legal modifications of $s$ for
$s \in {\mathbb S}_\kappa$, let $\mathbb{P}_\kappa = \{ [s] : s \in {\mathbb S}_\kappa \}$, and order
$\mathbb{P}_\kappa$ by ruling that $[t] \le [s]$ if and only if $\vert s \vert \le \vert t \vert$
and $[s] = [t \restriction \vert s \vert]$.
In this section we investigate $\mathbb{P}_\kappa$ as a notion of forcing, and 
 show that $\mathbb{P}_\kappa$ is a {\em $(\kappa+1)$-strategically closed} notion of forcing that adjoins a tree of the sort we desire.

\begin{definition}
Let $\mathbb{P}$ be a notion of forcing and let $\alpha$ be an ordinal.  The game $G_\alpha(\mathbb{P})$ involves two players, {\sf Odd} and {\sf Even}, who take turns playing conditions from $\mathbb{P}$ for $\alpha$ many moves. {\sf Odd} chooses their move at odd stages, and {\sf Even} chooses their move at even stages (including all limit stages). {\sf Even} is required to play $1_{\mathbb{P}}$ (the maximal element of $\mathbb{P}$) at move zero.   If $p_\beta$ is the condition played at move $\beta$, the player who played $p_\beta$ loses immediately unless $p_\beta\leq p_\gamma$ for all $\gamma<\beta$.  If neither player loses at any stage $\beta<\alpha$, then {\sf Even} wins the game.
\end{definition}

\begin{definition}
  Let $\mathbb{P}$ be a notion of forcing and $\gamma$ be an ordinal. The notion of forcing $\mathbb{P}$ is
  {\em  $\gamma$-strategically closed} if and only if {\sf Even} has a winning strategy in
  $G_{\gamma}(\mathbb{P})$.
\end{definition}

We come now to one of our main points.

\begin{theorem}
{}{$\mathbb{P}_\kappa$} is $(\kappa+1)$-strategically closed. 
\end{theorem}
\begin{proof}
The strategy for {\sf Even} in the game is simple, and involves building a capped sequence $\seq{ t_\beta:\beta<\kappa}$, where $t_\beta$ is an element of $p_{2\beta}$.  In the end, {\sf Even}'s victory will be assured by applying Lemma~\ref{closurelemma}.

Whenever {\sf Odd} chooses their move $p_{2\beta+1}$, {\sf Even} will choose a $t\in p_{2\beta+1}$ that is a $1$-extension of $t_\beta$, and define
$t_{\beta+1}=t^\smallfrown\seq{0}$
and 
$p_{2(\beta+1)}= [t_{\beta+1}]$.
At a limit stage $\delta\leq\kappa$, the capped sequence $\seq{t_\beta:\beta<\delta}$ will have a least upper bound $t_\delta$ in $\mathbb{S}_\kappa$, and {\sf Even} will then play the condition
$p_\delta = [t_\delta]$.
\end{proof}

Note that this game is very easy for {\sf Even} to win: if they are building a capped sequence
$\seq{t_\xi : \xi < \kappa}$ in the background, then all that is required at successor stages is that their response to $p_{2\beta+1}$ must contain a capped extension of $t_\beta$.  If this is done, then {\sf Even} will always be able to play at limit stages. This flexibility will be an important ingredient for us, as part of our proof relies on the fact that {\sf Even} has many winning moves available at successor stages.

The fact that $\mathbb{P}_\kappa$ is
{}{$(\kappa+1)$-strategically closed}
tells us that it {}{adds no $\kappa$-sequences of ordinals, and therefore}
preserves all cardinals up to and including $\kappa^+$.
If we assume $2^{\kappa}=\kappa^+$ as well, then all cardinals and cofinalities will be preserved.

The forcing also adds a $\kappa^+$-tree. Given a generic filter $G\subseteq\mathbb{P}_\kappa$,
let us step into the extension $\Vbf[G]$.  An easy density argument shows us that $G$ will consist
of a decreasing sequence $\seq{ p_\delta:\delta<\kappa^+ }$ of elements of $\mathbb{P}_\kappa$, which we enumerate so that $p_\delta$ consists of sequences of length $\delta+1$.

If we now define
$$
T(G):=\bigcup_{\delta<\kappa^+} p_\delta
$$
then it is straightforward to see that $T(G)$ forms a tree under extension.
Moreover, by construction $T(G)$ is $\r2$-coherent and full.

We will need to work with certain elementary submodels of cardinality $\kappa$. In the case when $\kappa$
is regular and $\kappa^{<\kappa} = \kappa$ we could use such models which are closed under sequences of length $<\kappa$,
but if $\kappa$ is singular this is impossible because in this case $\kappa^{\cf(\kappa)} > \kappa$,
and in any case we do not want to make  cardinal arithmetic assumptions. We will make a standard move and use
a certain type of ``internally approachable'' model.

If $\chi$ be a sufficiently large regular cardinal,
we will mildly abuse notation by writing ``$N \prec H(\chi)$'' as a shorthand for
  ``$N \prec (H(\chi), \in, \prec_\chi)$'' where $\prec_\chi$ is some fixed wellordering of $H(\chi)$. 
We claim that any parameter $x \in H(\chi)$, we can find an elementary submodel $M \prec H(\chi)$ satisfying the following:
\begin{itemize}

    \item $x \in M$;

    \item $M$ is of cardinality $\kappa$ with $\kappa + 1 \subseteq M$;

    \item $M\cap\kappa^+$ is some ordinal $\delta<\kappa^+$;

    \item for every $X \in M$ with $\vert X \vert \ge \kappa$, there is an enumeration $\vec x = \seq{x_i : i < \kappa}$
      of $X \cap M$ such that $\vec x \restriction j \in M$ for all
      $j < \kappa$. 

\end{itemize}
Let $\cf(\kappa) = \mu$. We construct $M$ as the union of a $\mu$-chain $(M_i)_{i < \mu}$
where:
\begin{itemize}

\item  $x, \kappa \in M_0$;

\item  for all $i < \mu$, $M_i \prec H(\chi)$ and $\vert M_i \vert < \kappa$;

\item  for all $i$ and $j$ with $i < j < \mu$, $M_i \subseteq M_j$ and $M_i \in M_j$;

\item  for all $j < \mu$, $\seq{M_i : i \le j} \in M_{j+1}$;

\item  for all $\gamma < \kappa$ there is $i < \mu$ such that $\gamma \subseteq M_i$. 

\end{itemize}
This is all possible if we choose $\chi$ sufficiently large. 

We verify that if we set $M := \bigcup_{i < \mu} M_i$ then $M$ is as required.
By construction $\kappa \subseteq M$, and so $M \cap \kappa^+ \in \kappa^+$. 
Now suppose $X \in M$ with $|X| \geq \kappa$.
To build $\vec x$, we assume without loss of generality that
$X \in M_0$. We start by choosing $\seq{x_i : i < \gamma_0}$ to be the
$<_\chi$-least enumeration of $X \cap M_0$, noting that $\gamma_0 < \kappa$ because $\vert M_0 \vert < \kappa$ and $\seq{x_i : i < \gamma_0} \in M_1$
because $X, M_0 \in M_1$. We will now proceed inductively for $\mu$ steps, choosing $\seq{x_i : i < \gamma_j}$ enumerating $X \cap M_j$ with $\gamma_j < \kappa$
and $\seq{x_i : i < \gamma_j} \in M_{j+1}$. Given $\seq{x_i: i < \gamma_j}$, we choose $\seq{y_i: i < \delta_j}$ to be the $<_\chi$-least enumeration of
$X \cap (M_{j+1} \setminus M_j)$, and then set $\gamma_{j+1} = \gamma_j + \delta_j$ and $x_{\gamma_j + i} = y_i$ for $i < \delta$.
Since
$X, M_j, M_{j+1} \in M_{j+2}$ it follows that $\seq{y_i: i < \delta_j} \in M_{j+2}$, and so $\seq{x_i: i < \gamma_{j+1}} \in M_{j+2}$. When $j$ is
limit let $\gamma_j = \sup_{j_0 < j} \gamma_{j_0}$, then $\gamma_j < \kappa$ because $j < \mu = \cf(\kappa)$, and
$\seq{x_i : i < \gamma_j} \in M_{j+1}$ because it can be defined from $\seq{M_i : i < j}$ and we have $\seq{M_i : i \le j} \in M_{j+1}$.

\begin{obs} If we require that the set $X \cap M$ be enumerated with repetitions, then we replace $X$ by $\kappa \times X$
and let $\seq{ (\alpha_i, x_i) : i < \kappa}$ be an enumeration of $(\kappa \times X) \cap M$ with all its proper initial segments in $M$.  
Then $\seq{ x_i : i < \kappa}$ enumerates $X$ with repetitions and all its proper initial segments lie in $M$.
\end{obs} 

Let $M$ be a submodel of this type, and note that since $\kappa \in M$ any set which is definable from the parameter $\kappa$ is also in $M$:
in particular the set ${\mathbb S}_\kappa$, the forcing poset $\mathbb{P}_\kappa$, the winning strategy for
 the game $G_{\kappa+1}(\mathbb{P}_\kappa)$,
the set of $\r2$-modifiers of length $\kappa^+$,
and the set of all dense subsets of $\mathbb{P}_\kappa$ are all elements of $M$.  
Given any $p\in M\cap\mathbb{P}_\kappa$, we can use our game $G_{\kappa+1}(\mathbb{P}_\kappa)$ to build an $(M, \mathbb{P}_\kappa)$-generic subset $G$ of $M\cap\mathbb{P}_\kappa$ together with a lower bound for $G$, that is to say a totally $(M, \mathbb{P}_\kappa)$-generic condition.
To this we fix an enumeration $\vec D$ of the dense subsets of $\mathbb{P}_\kappa$ which lie in $M$ in order type $\kappa$,
such that that every proper initial segment of $\vec D$  
is in $M$. We then build a run of the game $G_{\kappa+1}(\mathbb{P}_\kappa)$ where {\sf Even} uses the winning strategy, and player {\sf Odd}
plays by choosing $p_{2 \beta + 1}$ as the $<_\chi$-least extension of $p_{2 \beta}$ that lies in $D_\beta$. The key point is that for every
$\gamma < \kappa$, the sequence of moves up to $\gamma$ is defined from the strategy and an initial segment of $\vec D$, hence it is in $M$:
in particular $p_\gamma \in M$ for all $\gamma < \kappa$. It is now clear that the final move $p_\kappa$ is totally $(M, \mathbb{P}_\kappa)$-generic.
In particular, $p_\kappa$ induces an $(M, \mathbb{P}_\kappa)$-generic filter which determines our generic tree up to level $\delta=M\cap\kappa^+$,
and the same will occur if $\dot S$ is a name in $M$ for a subtree of $\dot T$.
We leverage this to establish that the generic tree $T(G)$ added by $\mathbb{P}_\kappa$ is such that all of its subtrees contain
a frozen cone.

\begin{theorem}
Every subtree of the generic tree $T(G)$ adjoined by $\mathbb{P}_\kappa$ contains a frozen cone.
Thus, there is a minimal non-$\sigma$-scattered linear order of cardinality $\kappa^+$ in the generic extension.
\end{theorem}
\begin{proof}
Let $\dot T$ be a $\Pbb_\kappa$-name for the generic tree $T(G)$, and suppose
\begin{equation}
\label{forces}
p\forces\text{``$\dot S$ is a subtree of $\dot T$ that does not contain a frozen cone''}.
\end{equation}
We will find $\delta<\kappa^+$ and $q\leq p$ such that
\begin{equation}
q\forces\text{``$\dot S\subseteq \dot T_{<\delta}$.''}
\end{equation}

Let $\chi$ be some sufficiently large regular cardinal, and let $M$ be an elementary submodel of $H(\chi)$ as discussed above, containing all parameters of interest to us. We let $\delta$ be $M \cap \kappa^+$.
As in the preceding discussion let $\vec D = \seq{D_i : i < \kappa}$ be an enumeration of the dense open subsets of $\mathbb{P}_\kappa$ that lie in $M$
with every proper initial segment of $\vec D$ in $M$, and let $\vec m = \seq{m_i : i < \kappa}$ be an enumeration with repetitions of the modifiers that lie in $M$
with every proper initial segment of $\vec m$ in $M$. 
We play the game $G_{\kappa+1}(\mathbb{P}_\kappa)$ to produce the required $q$.  The initial moves are as expected: {\sf Even} must open with $[\emptyset]$, and we let {\sf Odd} respond with $p$. 

Suppose now that we are playing the game, and it is {\sf Odd}'s turn to play.  In this situation, we have collaboratively built
$\seq{ p_\gamma:\gamma\leq 2\beta }$, while {\sf Even} has been building their auxiliary sequence $\seq{ t_\gamma:\gamma\leq\beta }$ on the side.
Our construction will be guided by $\vec D$ and $\vec m$, so that as in our prior construction of a totally generic condition
we have that $\seq{ p_\gamma:\gamma\leq 2\beta }$ and $\seq{ t_\gamma:\gamma\leq\beta }$ are both in $M$. The sequence $\seq{ t_\gamma:\gamma\leq\beta }$ will be topped, in particular
$\last(t_\beta) = 0$.

 We now consider the modifier $m_\beta$, noting that since $m_\beta \in M$ we have $h(m_\beta) < \delta$. 
 We ask first if $m_\beta$ is legal for $t_\beta$ with $h(m_\beta)<\Top(t_\beta)$.  If the answer is ``no,'' then {\sf Odd} doesn't need to take any special action,
 and chooses $p_{2\beta+1} := p_{2\beta}$. In this case {\sf Even} responds by choosing $t_{\beta + 1} := t_\beta$. 

 If the answer is ``yes,'' then we will ask {\sf Odd} to do some additional work.
 Note that $t_\beta + m_\beta \in M$ because $m_\beta, t_\beta \in M$.
 Also observe that the eventual constant value of $m_\beta$ is non-negative because
 $m_\beta$ is legal for $t_\beta$, $h(m_\beta) < \Top(t_\beta)$ and $\last(t_\beta)=0$.
 In particular $m_\beta$ is automatically legal for any extension of $t_\beta$.
 We choose $q \le p_{2 \beta}$ to be the $<_\chi$-minimal extension of $p_{2 \beta}$ deciding ``$t_\beta + m_\beta \in \dot S$.''
Since $q$ is definable from parameters in $M$, it is in $M$.
 If $q$ forces ``$t_\beta + m_\beta \notin \dot S$'' we let $p_{2\beta+1} = q$.
 In this case {\sf Even} choose $t_{\beta+1}$ as the $<_\chi$-least
 capped extension of $t_\beta$ with $[t_{\beta+1}] \le q$.

 If $q$ forces ``$t_\beta + m_\beta \in \dot S$'' we take $q$ and
 follow the procedure described above to extend it to a totally $(M, \mathbb{P}_\kappa)$-generic 
 condition, generating an $(M, \mathbb{P}_\kappa)$-generic filter $G_\beta$ on $M \cap \mathbb{P}_\kappa$.
 Of course $G_\beta$ itself is not in $M$, but we see shortly that this is not a problem.
 Using $G_\beta$ we can interpret names for the initial segments of $\dot S$ and $\dot T$ that are in $M$ and thus decide the identities of
 $T_{<\delta}$ and $S_{<\delta}$: these objects will depend on $G_\beta$, but for any $\alpha<\delta$ there will be a condition in $G_\beta$
 forcing that the information is valid through level $\alpha$. Since $q \in G_\beta$ we have that $t_\beta + m_\beta \in S_{<\delta}$.  

Since $p\in G_\beta$, our assumption (\ref{forces}) implies that for any
$s \in T_{<\delta}$ and $n < \omega$ there is an $n$-extension  $t$ of $s$ in $T_{<\delta}$ that is not in $S_{<\delta}$. This is the key ingredient of our argument.
Let $N$ be the norm of our modifier $m_\beta$. Since $t_\beta + m_\beta \in S_{<\delta}$,
$t_\beta + m_\beta$ has an $N$-extension $s$ in $T_{<\delta}$ such that $s \notin S_{<\delta}$.
This situation is forced to be true for this particular $s$ by some condition in $G_\beta$ which extends $q$. 

We have shown that there exist an $N$-extension $s$ of $t_\beta + m_\beta$ and an extension $r$ of $q$
such that $r$ forces ``$s \notin \dot S$.'' Let $(s', r')$ be the $<_\chi$-least pair with these properties, where as usual this pair is in $M$,
and let $p_{2 \beta + 1} = r'$.   We note that there is no reason to believe that $p_{2 \beta +1} \in G_\beta$ or that $s' = s$. 
Note also that we can just look at $s'$ and tell that it is an $N$-extension of $t_\beta + m_\beta$ without reference to the forcing at all,
so the point is that $p_{2\beta+1}$ contains enough information to determine that $s'$ is in $T$ but not in $S$.
This has some consequences, because the only way $p_{2\beta+1}$ can force $s'$ to be in $T$ is if $p_{2\beta+1}$ extends the equivalence class of $s'$ in $\mathbb{P_\kappa}$.  

By the closure properties of $p_{2 \beta + 1}$, $p_{2 \beta + 1}$ contains an $N$-extension $s''$ of $s'$.
By the definition of $T$,
$p_{2 \beta + 1}$ forces that $s'' \in T$ and since $S$ is forced to be downwards closed, $p_{2 \beta + 1}$ forces that $s'' \notin S$.
In summary, $p_{2\beta+1}$ contains an $N$-extension $s''$ of $t_\beta + m_\beta$
that is forced by $p_{2\beta+1}$ to lie outside of $S$.
 
Now define $t:=s'' -m_\beta$.
Since $N$ is the norm of $m_\beta$ and $s''$ is an $N$-extension of $t_\beta + m_\beta$, we know $-m_\beta$ is legal for $s''$
and $t$ will be a $1$-extension of $t_\beta$.
Now {\sf Even} defines
$$
t_{\beta+1}:= t^\smallfrown\seq{ 0 }
$$
and
$p_{2(\beta+1)}:=[t_{\beta+1}]$, and play continues. As we observed above, $m_\beta$ is legal
for $t_{\beta+1}$.

We summarise the results of this round of the construction, keeping in mind that there were various cases.
We claim that in all cases where $m_\beta$ is legal for $t_\beta$ with $h(m_\beta)<\Top(t_\beta)$, 
$t_{\beta+1}$ is a capped extension of $t_\beta$ and
$$
p_{2(\beta+1)}\forces\text{``$t_{\beta+1} + m_\beta \notin \dot S$.''}
$$
If we are in the case where $q$ forces ``$t_\beta + m_\beta \notin \dot S$,'' then we set $p_{2\beta+1} := q$ 
and the claim is immediate because $t_{\beta + 1} + m_\beta$ extends $t_\beta + m_\beta$.  
If we are in the case where $q$ forces ``$t_\beta + m_\beta \in \dot S$,'' then we
arranged that $t_{\beta +1} + m_\beta$ extends $s''$ and that $p_{2 \beta + 1}$ forces ``$s'' \notin S$.''     

Because we were careful to make all choices at the successor stages using the wellordering $<_\chi$,
$\seq{ p_\gamma:\gamma\leq 2\beta}$ and $\seq{ t_\gamma:\gamma\leq\beta }$ are both in $M$ for all $\gamma < \kappa$. 
If {\sf Even} follows this strategy, then they will end up winning the game by Lemma~\ref{closurelemma},
because the sequence $\seq{ t_\beta:\beta<\kappa }$ is a capped sequence. Let $q$ be the corresponding final move $p_\kappa$ for {\sf Even}, and now we claim
$$
q\forces\text{``$\dot S \subseteq\dot T_{<\delta}$.''}
$$
To see this, let us define
$$
t:= \Big(\bigcup_{\beta<\kappa}t_\beta \Big) \cat \seq{ 0 }.
$$
Observe that $t \in \Sbb_\kappa$ is a bound of the capped sequence $\seq{ t_\beta:\beta<\kappa }$ that {\sf Even} built during our run of the game.
We know $q=[t]$, so it suffices to show for any $\r2$-modifier $m$ that is legal for $t$ that
$$
    q\forces\text{``$t+m \notin\dot S$.''}
$$ 

It suffices to check this for modifiers $m$ that are in $M$, as $t+m$ is completely determined by $m \restr \delta$ and $m$ must be constant on a tail of $\delta$.
Since we enumerated the modifications in $M$ with repetitions, during our play of the game we came to a stage $2\beta+1$ for which 
$m_\beta = m$ and
$\h(m_\beta)<\Top(t_\beta)$.
Since $m$ is legal for $t$, we know $m$ is legal for $t_\beta$ and therefore $t_{\beta+1}$ was selected so that
$$
p_{2(\beta+1)}\forces\text{`` $t_{\beta+1}+m_\beta \notin\dot S$.''}
$$
Hence
$$
q\forces\text{`` $t+m\notin\dot S$''}
$$
as required.
\end{proof}

\section{Building many examples}

\label{axiomaic:sec}

Our goal in this section is to prove that if $\Vbf=\Lbf$ then there is a minimal non-$\sigma$-scattered
linear order of cardinality $\kappa^+$ for every infinite cardinal $\kappa$.
This will be achieved by showing that such linear orders can be constructed using the ``diamond in the square'' principle
$\sq_\kappa$.
Principles of this type, which combine $\diamondsuit_{\kappa^+}$ and $\square_\kappa$ for some infinite cardinal $\kappa$, were
first introduced by Gray~\cite{graythesis} for $\kappa = \omega_1$ in his dissertation,
and first appeared in the literature in work of Abraham, Shelah, and Solovay~\cite{ASS}.
The paper \cite{ASS} develops several ``diamond in the square'' principles: the principle now
known as $\sq_\kappa$ appears there in a slightly different (but equivalent) form under the name $SD_\kappa$.  
If $\Vbf = \Lbf$, then $\sq_\kappa$ holds for every infinite cardinal $\kappa$  \cite[Section 2]{ASS}.

We recall the definition of $\sq_\kappa$.
If $C$ is a set of ordinals, let $\acc(C)$ denote the set of elements of $C$ which are limit points of $C$.

\begin{definition}
The principle $\sq_\kappa$ asserts the existence of a sequence $$\seq{ (C_\delta, X_\delta):\delta<\kappa^+ }$$ such that:
\begin{enumerate}

    \item for limit $\delta<\kappa^+$ the set $C_\delta$ is a closed unbounded subset of $\delta$ of order-type at most $\kappa$,

    \item $X_\delta\subseteq\delta$ for all $\delta<\kappa^+$,

    \item if $\alpha\in\acc(C_\delta)$ then: 

    \begin{itemize}

        \item  $C_\alpha = C_\delta\cap\alpha$,

        \item $X_\alpha = X_\delta\cap\alpha$,

    \end{itemize}

    \item for every subset $X\subseteq\kappa^+$ and every club $C\subseteq\kappa^+$ there is a limit ordinal $\delta\in C$ such that: 
       \begin{itemize}

        \item $C_\delta \subseteq C$,

        \item $\otp(C_\delta)=\kappa$, and

        \item $X\cap \delta = X_\delta$.

    \end{itemize}
\end{enumerate}
\end{definition}

We will need the following lemma due to Assaf Rinot; see Remark \ref{Rinot_rem} below.

\begin{lemma}
\label{sqplus}
Suppose that $\vec{C}: = \seq{ C_\delta : \delta < \kappa^+ }$ is a $\square_\kappa$-sequence. 
Then there exists a sequence $\seq{ f_\delta : \delta<\kappa^+ }$ of functions $f_\delta : C_\delta \to \delta$
such that for every limit ordinal $\delta < \kappa^+$:
\begin{itemize}

\item for every $\gamma \in\acc(C_\delta)$, $f_\gamma=f_\delta\restriction\gamma$, and 

\item if $\otp(C_\delta)=\kappa$, then $f_\delta$ maps $C_\delta$ onto $\delta$.

\end{itemize}
\end{lemma}

\begin{proof}
This may be extracted from the proof of \cite[Lemma~3.8]{paper28}, but we prove this simplified case from scratch.
Fix a map $e:[\kappa^+]^2\rightarrow\kappa$ such that if $\alpha < \beta < \gamma$ then $e(\alpha,\gamma) \ne e(\beta,\gamma)$.
Let $\pi:\kappa \rightarrow \kappa \times \kappa$ be a surjection such that the preimage of any singleton is cofinal in $\kappa$.
For every $\delta\in\acc(\kappa^+)$, define a function $f_\delta:C_\delta\rightarrow\delta$, by letting for all $\beta\in C_\delta$:
$$
f_\delta(\beta):=
\min(\{\alpha<\beta :
\pi(\otp(C_\delta\cap\beta))=(\otp(C_\delta\cap\alpha),e(\alpha,\min (C_\delta\setminus(\alpha+1))))
\}\cup\{\beta\}).$$
\begin{claim} Let $\delta\in\acc(\kappa^+)$ and $\gamma\in\acc(C_\delta)$. Then $f_\gamma=f_\delta\restriction\gamma$.
\end{claim}
\begin{proof} As $\vec C$ is a $\square_\kappa$-sequence, $C_\gamma=C_\delta\cap\gamma$.
So, for all $\alpha<\beta<\gamma$, $C_\delta\cap\beta=C_\gamma\cap\beta$, 
$C_\delta\cap\alpha=C_\gamma\cap\alpha$ and $C_\delta\setminus(\alpha+1)=C_\gamma\setminus(\alpha+1)$.
Consequently, $f_\gamma=f_\delta\restriction\gamma$.
\end{proof}
\begin{claim} Let $\delta\in\acc(\kappa^+)$ with $\otp(C_\delta)=\kappa$. Then $f_\delta$ maps $C_\delta$ onto $\delta$.
\end{claim}
\begin{proof} Let $\alpha<\delta$. Set $\eta:=\min(C_\delta\setminus(\alpha+1))$ and $(i,j):=(\otp(C_\delta\cap\alpha),e(\alpha,\eta))$.
By the choice of the surjection $\pi$, $\{ \varepsilon < \kappa : \pi(\varepsilon)=(i,j)\}$ is a cofinal subset of $\otp(C_\delta)$,
so we may fix some $\beta\in C_\delta$ above $\alpha$ such that $\pi(\otp(C_\delta\cap\beta))=(i,j)$.
By the definition of $f_\delta$, it now follows that if $f_\delta(\beta)\neq\alpha$, then there exists some $\bar\alpha<\alpha$ such that 
$\otp(C_\delta\cap\bar\alpha)=i$ and $e(\bar\alpha,\min(C_\delta\setminus(\bar\alpha+1)))=j$.
Towards a contradiction, suppose that $\bar\alpha<\alpha$ is such an ordinal.
But as $\otp(C_\delta\cap\bar\alpha)=i=\otp(C_\delta\cap\alpha)$, it is the case that
$\min(C_\delta\setminus(\bar\alpha+1))=\min(C_\delta\setminus(\alpha+1))=\eta$,
so, $e(\bar\alpha,\eta)=j=e(\alpha,\eta)$, contradicting the fact that the fiber $e(\cdot,\eta)$ is injective.
\end{proof}
This completes the proof.
\end{proof}

The following formal strengthening of $\sq_\kappa$---which is an equivalent by Lemma \ref{sqplus}---willl be useful in carrying out our construction below.

\begin{definition}
If $\seq{(C_\delta,X_\delta) : \delta < \kappa^+}$ is a $\sq_\kappa$-sequence and $\seq{f_\delta : \delta < \kappa^+}$
satisfies the conclusion of Lemma \ref{sqplus} with respect to $\seq{C_\delta : \delta < \kappa^+}$,
then we say that $\seq{(C_\delta,X_\delta,f_\delta) : \delta < \kappa^+}$ is a \emph{$\sq_\kappa^{+\epsilon}$-sequence}.
The hypothesis $\sq_\kappa^{+\epsilon}$ postulates the existence of a $\sq_\kappa^{+\epsilon}$-sequence.
\end{definition}

\begin{remark} \label{Rinot_rem}
By Lemma \ref{sqplus}, $\sq_\kappa$ implies $\sq_\kappa^{+\epsilon}$.
The first circulated draft of this paper derived $\sq_\kappa^{+\epsilon}$ from the stronger hypothesis $\sq_\kappa^*$ introduced by Rinot and Schindler in \cite{rinotschindler} and which also holds if $\Vbf = \Lbf$.
Following the third author's presentation of the results of this paper in the Toronto Set Theory Seminar in February 2023,
Rinot informed us that $\sq_\kappa^{+\epsilon}$ already followed from $\sq_\kappa$.
He has generously given us permission to include above his formulation of Lemma \ref{sqplus} and its proof.
\end{remark}

We now turn to the task of using $\sq_\kappa$ to construct a minimal non-$\sigma$-scattered linear order of cardinality $\kappa^+$.

\begin{theorem}
\label{construct}
If $\kappa$ is an infinite cardinal for which $\sq_\kappa$ holds,
then there is a $T \subseteq \Sbb_\kappa$ which is $\r2$-coherent and full.
Consequently, there is a minimal non-$\sigma$-scattered linear ordering of cardinality $\kappa^+$ which is moreover
$\kappa^+$-Countryman.
\end{theorem}
\begin{proof}
Applying Lemma \ref{sqplus}, let $\seq{ (C_\delta, X_\delta, f_\delta):\delta<\kappa^+ }$ be a $\sq_\kappa^{+\epsilon}$-sequence,
and fix an enumeration $\seq{ m_\delta:\delta<\kappa^+ }$ of all $\r2$-modifiers,
subject to the conditions that $\h(m_\delta)<\delta$ and that each modifier appears in the enumeration unboundedly often.

We need to give a little attention to how we use our $\sq_\kappa^{+\epsilon}$-sequence to guess $\kappa^+$-trees.
This will be done in a completely straightforward way, but at one point in the proof the specificity will be convenient.
Since $\diamondsuit(\kappa^+)$ is a consequence of $\sq_\kappa^{+\epsilon}$,
we know $\kappa^\kappa=\kappa^+$ and so we can fix an enumeration $\seq{ \sigma_\alpha:\alpha<\kappa^+ }$ of $\mathbb{S}_\kappa$ in order-type $\kappa^+$.
Given any $\kappa^+$-tree $S\subseteq \mathbb{S}_\kappa$ we can code $S$ with a set $X\subseteq\kappa^+$ by setting
\[
X:=\{\alpha<\kappa^+:\sigma_\alpha\in S\}.
\]
What we need to observe is that if we do this, then there will be a closed unbounded set of $\delta<\kappa^+$ for which
\begin{equation}
\label{code}
S_{<\delta} = \{\sigma_\alpha:\alpha\in X\cap\delta\}.
\end{equation}
This observation will help us later when we try to apply $\sqdiamond_\kappa^{+\epsilon}$.

The tree is built via a construction of length $\kappa^+$, and we build a sequence $\seq{ t_\alpha:\alpha<\kappa^+ }$ of elements of $\Sbb_\kappa$ with $\Top(t_\alpha)=\alpha$ that further satisfy
$$t_\beta\equiv_\kappa t_\alpha\restr\beta+1$$
whenever $\beta<\alpha<\kappa^+$.
At a typical stage $\alpha$ of our construction, we will have available the sequence $\seq{ t_\beta:\beta<\alpha }$ (hence we will know $T_{<\alpha}$) and will need to produce a suitable $t_\alpha$ with domain $\alpha+1$.   

The particular choice of $t_\alpha$ will matter only in cases where $\alpha$ is a limit ordinal, because if $\alpha$ is a successor ordinal $\gamma+1$ then we set
\[
t_\alpha:=t_\gamma\vphantom{b}\!^\smallfrown\seq{ 0 }.
\]

At a limit stage $\alpha$ of our construction, we commit to building a $t_\alpha \in \Sbb_\kappa$
which corresponds to a cofinal branch through $T_{<\alpha}$ and satisfies the following two conditions:
\begin{equation}
\label{one}
    \beta\in\acc(C_\alpha)\Longrightarrow t_\beta\subseteq t_\alpha, 
\end{equation}
and
\begin{equation}
\label{needed}
\acc(C_\alpha)\subseteq t^{-1}_\alpha(\{0\}) \subseteq \acc(C_\alpha)\cup\{\beta+1:\beta\in\nacc(C_\alpha))\}.
\end{equation}
Notice that this last condition will guarantee that the set of $\beta<\alpha$ for which $t_\alpha(\beta)=0$ will have order-type at most $\kappa$.
Since $\alpha$ is a limit ordinal, membership of $t_\alpha$ to $\Sbb_\kappa$ requires a condition along the lines of (\ref{needed}) to allow us to define $t_\alpha(\alpha)=0$.

We have no freedom if $\acc(C_\alpha)$ happens to be unbounded in $\alpha$,
as (\ref{one}) will force us to define
\[
t_\alpha := \Big( \bigcup_{\beta\in\acc(C_\alpha)} t_\beta \Big) \cat \seq{0 },
\]
and this will be an element of $\Sbb_\kappa$ with the required properties. Thus, the only leeway in our construction occurs when the set $\acc(C_\alpha)$ is bounded below $\alpha$, and whatever substantive action we take must occur at these stages.

Suppose then that our construction has arrived at a limit ordinal~$\alpha$ for which $\gamma:=\sup(\acc(C_\alpha))$ is less than~$\alpha$.
In such a situation, we know that $C_\alpha\setminus{\gamma+1}$ must have order-type $\omega$, so we can list it in increasing order as $\seq{ \alpha_n:n<\omega }$.  When we choose $t_\alpha \in \Sbb_\kappa$, we will want to make sure that it
satisfies the following {\em structural requirements}:
\begin{itemize}

    \item $\Top(t_\alpha)=\alpha$,

    \item $t_\gamma\subseteq t_\alpha$

    \item $t_\alpha\restr\beta+1\in T_{<\alpha}$ for all $\beta<\alpha$, and 

    \item there is an $m<\omega$ such that 
    \[
    t_\alpha^{-1}(\{0\})\cap (\gamma,\alpha)=\{\alpha_n+1:m\leq n<\omega\}.
    \]
\end{itemize}

As long as $t_\alpha$ satisfies these requirements, our construction can proceed. 
They are not difficult to arrange:
if $s$ is any $1$-extension of $t_\gamma$ in $T_{<\alpha}$ at all,
then we can extend $s$ to a suitable $t_\alpha$ by means of a capped sequence of length $\omega$
whose tops consist of the ordinals $\alpha_n+1$ for $m\leq n<\omega$.

Our work at stage $\alpha$ will depend on the set $X_\alpha$ presented to us by the $\sq_\kappa^{+\epsilon}$-sequence. Let us agree to call $\alpha$ an {\em active} stage if the following three criteria are satisfied:
\begin{itemize}

\item $X_\alpha$ codes an unbounded subtree $Y_\alpha$ of $T_{<\alpha}$,

\item $Y_\alpha$ does not contain a frozen cone of $T_{<\alpha}$, and

\item there is a $\xi\in C_\gamma$ for which $t_\gamma + m_{f_\gamma(\xi)}$ is in $Y_\alpha$.

\end{itemize}
If $\alpha$ is an active stage, then let $\zeta\in C_\gamma$ be the least $\xi$ as above.
We say that this $\zeta$ is {\em targeted for action} at stage $\alpha$,
and our task will be to find an extension $t_\alpha$ of $t_\gamma$ that satisfies all the structural requirements with the additional property that
\begin{equation}
\label{stagegoal}
(t_\alpha + m_{f_\gamma(\zeta)})\restr\alpha\text{ is not a cofinal branch through }Y_\alpha.
\end{equation}
If on the other hand $\alpha$ is not an active stage, then we can simply let $t_\alpha \in \Sbb_\kappa$
be any extension of $t_\gamma$ that satisfies the structural requirements.

Suppose now that $\alpha$ is an active stage, and $\zeta\in C_\gamma$ is the corresponding target.
It suffices to produce a 1-extension $s$ of $t_\gamma$ in $T_{<\alpha}$ with the property that $s+m_{f_\gamma(\zeta)}$ is not in $Y_\alpha$,
as such an $s$ can be extended to the $t_\alpha$ we need.
To do this, let $N$ be the norm of the modifier $m_{f_\gamma(\zeta)}$.
Since $Y_\alpha$ does not contain a frozen cone of $T_{<\alpha}$,
we know that $t_\gamma + m_{f_\gamma(\zeta)}$ has an $N$-extension $t$ in $T_{<\alpha}$ that is not in $Y_\alpha$.
By definition, the modifier $-m_{f_\gamma(\zeta)}$ will be legal for $t$, and 
\[
s:=t-m_\zeta
\]
will be a $1$-extension of $t_\gamma$ in $T_{<\alpha}$ of the sort we seek, and therefore we can find $t_\alpha$ which satisfies (\ref{stagegoal}) in addition to the structural requirements. This completes stage $\alpha$.

Why does this construction succeed?  We let $T$ be the $\kappa^+$-tree determined by our sequence $\seq{ t_\alpha:\alpha<\kappa^+ }$, so that level $\alpha$ of $T$ will consist of all the legal modifications of $t_\alpha$.  Our task is to show that any unbounded subtree of $T$ contains a frozen cone, so assume by way of contradiction that $S\subseteq T$ is a counterexample, and let $X\subseteq\kappa^+$ code $S$.

There is a closed unbounded set $E$ of ordinals $\delta<\kappa^+$ satisfying the following two statements:
\begin{itemize}

\item if $s$ is in $S_{<\delta}$ and $n<\omega$, then $s$ has an $n$-extension in $T_{<\delta}$ that is not in $S$;

\item if $\delta\in E$ then $S_{<\delta}$ is coded by $X\cap\delta$.

\end{itemize}
Notice that this last is where we use the property of our coding mechanism discussed in the context of (\ref{code}).

If $\chi$ is some sufficiently large regular cardinal, we can find an elementary submodel $M$ of $H(\chi)$ of cardinality $\kappa$ that contains $S$, $T$, and $E$ such that: 
\begin{itemize}

\item $M\cap\kappa^+=\delta<\kappa^+$,

\item $C_\delta\subseteq E$,

\item  $X_\delta = X\cap \delta$, and

\item $f_\delta$ maps $C_\delta$ onto $\delta$.

\end{itemize}
This can be achieved because of the properties of our $\sq^{+\epsilon}_\kappa$-sequence: note that the definition implies that there will be a stationary set of $\delta$ satisfying the last three requirements above, hence we can find $\delta$ satisfying the first.

Since $S$ contains an element from level $\delta$ of $T$, there is at least one legal modification of $t_\delta$ in $S$.  Since $\delta$ is a limit ordinal, we may assume that the relevant modifier $m$ satisfies $\h(m)<\delta$, and hence the modifier $m$ will be in the model $M$ and therefore will appear before stage $\delta$ in our enumeration of $\mathbb{S}_\kappa$.

Since the function $f_\delta$ maps $C_\delta$ onto $\delta$, the modifier $m$ guarantees that there is some least $\zeta\in C_\delta$ for which 
\begin{equation}
\label{cofinalbranch}
t_\delta + m_{f_\delta(\zeta)}\restr\delta\text{ is a cofinal branch through }S_{<\delta}.
\end{equation}

Now turn our focus to the way our construction proceeds through the stages indexed by $\acc(C_\delta)$.   Suppose now that $\alpha$ is in $\acc(C_\delta)$.   By the coherence of our $\sq_\kappa^{+\epsilon}$-sequence, we know
\[
X_\alpha = X_\delta\cap\alpha = X\cap\alpha
\]
and since $\alpha$ is also in $E$, we conclude that $X_\alpha$ codes $S_{<\alpha}$.  We also know that $S_{<\alpha}$ does not contain a frozen cone of $T_{<\alpha}$, as this fact will reflect to $\alpha$ by our choice of $E$.   Thus, any $\alpha\in \acc(C_\delta)$ will satisfy the first two requirements needed to be an active stage of our construction. 

We now show that all sufficiently large elements of $\nacc(\acc(C_\delta))$ will be active stages of our construction. More specifically, if $\alpha\in \acc(C_\delta)$ and
\[
\zeta<\gamma:=\sup(\acc(C_\delta)\cap\alpha)<\alpha,
\]
then $\alpha$ will satisfy the third requirement of being an active stage of our construction.  To see this, note that since we are working with a $\sq^{+\epsilon}_\kappa$-sequence we have

\begin{align}
\zeta\in C_\gamma = C_\delta\cap\gamma,\\
\intertext{and}
f_\gamma(\zeta)=f_\delta(\zeta).
\end{align}
Since $t_\delta + m_{f_\delta(\zeta)} \restr \delta$ is a cofinal branch through through $S_{<\delta}$,  we know
\[
t_\gamma + m_{f_\gamma(\zeta)} \restr \gamma \text{ is a cofinal branch through $S_{<\gamma}$},
\]
and therefore $\alpha$ must be an active stage of the construction.

Said another way, we have shown that all sufficiently large $\alpha\in\nacc(\acc(C_\delta))$ are active stages. This is enough to get a contradiction: since $\otp(C_\delta)=\kappa$ we know
\[
\otp(\nacc(\acc(C_\delta))\setminus (\gamma+1)) = \kappa,
\]
and our construction guarantees that once an ordinal has been targeted at such a stage $\alpha$, it will never be targeted again at any future stage from $\acc(C_\delta)$. Thus, we must eventually arrive at an active stage $\alpha\in C_\delta$ where $\zeta$ will be targeted for action, but the choice of $t_\alpha$ then contradicts (\ref{cofinalbranch}).  We conclude that $S$ must contain a frozen cone, and the theorem is established.
\end{proof}

\section{Concluding remarks}

\label{limitations:sec}

We feel that it is likely that the methods of this paper can be adapted to show that in $\Lbf$, there is a $\kappa$-Aronszajn line which
is minimal with respect to being non-$\sigma$-scattered whenever $\kappa$ is an uncountable regular cardinal
which is not weakly compact.
Presumably if $\kappa$ is regular uncountable and not weakly compact,
$\Lbf$ satisfies a suitable principle $\sqdiamond(\kappa)$, which in turn implies that there is a $\kappa$-Aronszajn tree $T \subseteq \omega^{<\kappa}$ with the following properties:
\begin{itemize}
    \item $T$ is $\r2$-coherent and full;
    \item every subset of $T$ of cardinality $\kappa$ contains an antichain of cardinality $\kappa$;
    \item every subtree of $T$ contains a frozen cone.
\end{itemize}
The arguments presented in this paper then show that the lexicographic ordering on any antichain in $T$ of cardinality $\kappa$
is minimal with respect to not being $\sigma$-scattered.

Galvin asked whether there is a minimal non-$\sigma$-scattered linear order with the additional property that every uncountable suborder contains a copy of $\omega_1$---this is equivalent to being minimal with respect to not being a countable union of 
well orders (see \cite[Problem 4]{new_class_otp}).
As noted in the introduction, Ishiu and the third author have shown that a negative answer follows from $\PFA^+$ \cite{ishiu-moore}
and Lamei Ramandi has shown that a positive answer is consistent \cite{lamei-ramandi2}.
It remains an open problem whether there are consistent examples of linear orders of cardinality greater than $\aleph_1$
which are minimal with respect to not being a countable union of well orders.
Such orders necessarily are not $\kappa$-Aronszajn and hence the methods of this paper do not seem to shed much light on this problem.
Todorcevic has shown that $\Box_{\aleph_\omega}$ implies that there is a linear order of cardinality $\aleph_{\omega+1}$ of density
$\aleph_\omega$ such that every suborder of cardinality $\aleph_\omega$ is a countable union of well orders \cite[7.6]{walksbook}.
Note, however, that the construction of Dushnik and Miller \cite{dushnik-miller} generalizes to show that if $2^\kappa = \kappa^+$,
then there is no minimal linear order of cardinality $\kappa^+$ and density $\kappa$.
Thus at least consistently, Todorcevic's example \cite[7.6]{walksbook} does not solve Galvin's problem;
one would need an analog of Baumgartner's model \cite{reals_iso} at the level of $\aleph_{\omega +1}$, which seems beyond the reach of current methods.

A minimal non-$\sigma$-scattered ordering of cardinality $\lambda > \aleph_1$
is a ``non-reflecting'' object, in the sense that it enjoys a property which is not
enjoyed by any of its properly smaller suborderings.
This phenomenon is ruled out by large cardinal assumptions.
For instance if $\lambda$ is weakly compact, then any non-$\sigma$-scattered order of cardinality $\lambda$
has a non-$\sigma$-scattered suborder of smaller cardinality.
Similarly if $\kappa$ is supercompact, then any non-$\sigma$-scattered linear order has a non-$\sigma$-scattered suborder of cardinality less than $\kappa$.
The proofs of these statements are routine modifications of arguments in \cite[\S7]{new_class_otp}.

\end{document}